\documentclass[12pt]{article}

\usepackage[top=1in, bottom=1in, left=1in, right=1in]{geometry}
\usepackage{amsmath}
\usepackage{amsthm}
\usepackage{amssymb}
\usepackage[flushmargin]{footmisc}
\usepackage[colorlinks]{hyperref}
\allowdisplaybreaks
\hypersetup{
linkcolor=blue,      
citecolor=blue,    
}

\newtheorem{lemma}{Lemma}[section]
\newtheorem{proposition}{Proposition}[section]
\newtheorem{theorem}{Theorem}[section]

\theoremstyle{definition}
\newtheorem{definition}{Definition}[section]

\theoremstyle{remark}
\newtheorem{remark}{Remark}[section]

\newcommand{\ad}{\operatorname{ad}}
\newcommand{\Ad}{\operatorname{Ad}}
\newcommand{\SL}{\operatorname{SL}}
\newcommand{\SO}{\operatorname{SO}}

\newcommand{\Stab}{\operatorname{Stab}}
\newcommand{\Lie}{\operatorname{Lie}}

\begin{document}
\title{A Shrinking Target Problem with Target at Infinity in Rank One Homogeneous Spaces}
\author{Cheng Zheng}
\date{}
\maketitle

\let\thefootnote\relax\footnotetext{Mathematics Department, Technion-Israel Institute of Technology, Haifa, 32000, Israel\\ Email: cheng.zheng@campus.technion.ac.il}

\vspace{-0.2in}
\begin{abstract}
In this paper, we give a definition of Diophantine points of type $\gamma$ for $\gamma\geq0$ in a homogeneous space $G/\Gamma$ and compute the Hausdorff dimension of the subset of points which are not Diophantine of type $\gamma$ when $G$ is a simple Lie group of real rank one. We also deduce a Jarnik-Besicovitch Theorem on Diophantine approximation in Heisenberg groups.
\end{abstract}

\section{Introduction and main results}\label{intro}
\subsection{Introduction}
The classical Jarn\'ik-Besicovitch theorem \cite{Bes,J} in Diophantine approximation theory says that for any $\gamma\geq1$ $$\dim_H\left\{x\in\mathbb R: \exists C>0 \text{ such that } \left|x-\frac mn\right|\geq\frac C{n^{\gamma+1}}\left(\forall\frac mn\in\mathbb Q\right)\right\}^c=\frac2{1+\gamma}$$ where $\dim_H$ denotes the Hausdorff dimension with respect to the Euclidean distance and $S^c$ the complement of a subset $S$ in a space. Using a generalized version of Dani's correspondence established by Kleinbock and Margulis \cite{D0, KM1}, one could reformulate this theorem as follows. Let $X_2=\SL(2,\mathbb R)/\SL(2,\mathbb Z)$. Then for any $\gamma\geq1$ 
\begin{align*}
\dim_H\left\{p\in X_2:\exists C>0 \text{ such that }\delta(a_tp)\geq Ce^{-\frac{\gamma-1}{\gamma+1} t}(\forall t>0)\right\}^c=2+\frac2{1+\gamma}.
\end{align*}
Here $\dim_H$ denotes the Hausdorff dimension with respect to a metric on $X_2$ induced by a norm on the Lie algebra of $\SL(2,\mathbb R)$, $\delta(\cdot)$ is the function defined on the space of unimodular lattices $\Lambda\subset\mathbb R^2$ by $$\delta(\Lambda)=\inf_{v\in\Lambda\setminus\{0\}}\|v\|$$ and $a_t=\text{diag}(e^t,e^{-t})$ is the diagonal group in $\SL(2,\mathbb R)$. Note that $\delta(a_tp)$ measures the excursion rate of an orbit $\{a_tp\}$ in $\SL(2,\mathbb R)/\SL(2,\mathbb Z)$.

This Jarn\'ik-Besicovitch theorem has been generalized in various cases. For instance, Meli\'an and Pestana \cite{MP} obtain a formula for Hausdorff dimensions of subsets of points with different geodesic excursion rates in a finite-volume hyperbolic manifold. As a consequence, their result implies a Jarn\'ik-Besicovitch theorem on Diophantine approximation in some quadratic number fields. Also the work by Dodson \cite{D1} describes Hausdorff dimensions of subsets of non-Diophantine matrices with different orders, which could be rephrased for points in the homogeneous space $\SL(n,\mathbb R)/\SL(n,\mathbb Z)$ with different excursion rates under some semisimple flow \cite{KM1}.

Note that the examples above are related to a shrinking target problem proposed by Hill and Velani \cite{HV1}. Specifically, let $f:X\to X$ be a continuous transformation on a metric space $X$. Consider the subset $$\{x\in X: f^n(x)\in B(x_0,r(n))\text{ for infinitely many }n\}$$ where $x_0\in X$,  $r:\mathbb N\to\mathbb R$ is a decreasing function and $B(x_0,r(n))$ denotes the ball of radius $r(n)$ around $x_0$. This subset collects the points whose orbits under $f$ hit a shrinking target infinitely many times. One could then ask what is the Hausdorff dimension of this subset, or generally what is the size of this subset. The example of the geodesic flow $a_t=\text{diag}(e^t,e^{-t})$ on $\SL(2,\mathbb R)/\SL(2,\mathbb Z)$ mentioned in the beginning fits an analogue of this problem quite well (similarly for \cite{D1,MP}), where the target $x_0$ is the cusp $\infty$ of $\SL(2,\mathbb R)/\SL(2,\mathbb Z)$, and the subset $$\left\{p\in \SL(2,\mathbb R)/\SL(2,\mathbb Z):\exists C>0 \text{ such that }\delta(a_tp)\geq Ce^{-\frac{\gamma-1}{\gamma+1} t}(\forall t>0)\right\}^c$$ collects the points whose orbits under the geodesic flow approach the cusp infinitely often with certain rate. 

The shrinking target problem is studied by Hill and Velani for the expanding rational maps of the Riemann sphere on Julia sets \cite{HV1,HV2}. The approach in \cite{HV1,HV2} is later developed, e.g. by Urba\'nski \cite{U1} for conformal iterated function systems and Reeve \cite{R11} for the expanding Markov maps on intervals. For geodesic flows on compact manifolds with negative curvature, one could refer to \cite{HP1} by Hersonsky and Paulin.

In this paper, we consider the shrinking target problem with target at $\infty$ in the homogeneous space $G/\Gamma$ where $G$ is a simple Lie group of real rank one and $\Gamma$ is a non-uniform lattice in $G$. Let $\{a_t\}$ denote a semisimple flow (i.e., every element in the one-parameter subgroup $\{a_t\}$ is Ad-semisimple in $G$) on $G/\Gamma$. According to the viewpoint of \cite{HV1}, we would consider the dynamical system $$a_t:G/\Gamma\to G/\Gamma$$ and study the subset of points whose trajectories under this flow enter into cusps infinitely often with certain rate. Eventually, we will derive an explicit formula for the Hausdorff dimension of such a subset, which generalizes \cite{MP}. We will then deduce a Jarn\'ik-Besicovitch theorem on Diophantine approximation in Heisenberg groups.

We remark that this question in the context of geodesic flows on negatively curved manifolds with cusps is discussed by Hersonsky and Paulin \cite{HP2,HP3,HP4}, where they establish a Khintchine-Sullivan type theorem showing a necessary and sufficient condition for a subset of Diophantine points having full or null measure \cite{HP4}. Here we discuss a conceptually further question  about Hausdorff dimensions of null Diophantine subsets.

Our proof is mainly based on the classical method of constructing Cantor-type subsets (cf.\cite{HV2, KM}). We reduce our problem to computing the Hausdorff dimension of a limsup subset of open boxes via the structure theory of Lie groups. It turns out that in some case, these open boxes are not cubes, but rather rectangles with different side lengths due to the distinct positive roots (or different expanding rates) of the semisimple flow $\{a_t\}$ on $G/\Gamma$. This non-conformal phenomenon is quite different from \cite{HV1, HV2, MP, R11, U1}, and if one uses formulas for Hausdorff dimensions of limsup subsets of balls (for example \cite{BV,DRV}), it seems difficult to obtain a sharp lower bound. Instead, we divide these boxes further into cubes. Based on these new cubes, we would be able to construct a Cantor-type subset, and by a result of McMullen \cite{M} and Urba\'nski \cite{U}, one could obtain the desired lower bound. The calculation for the lower bound also relies on a counting problem regarding the distribution of some set consisting of rational points, which is related to the mixing property of the semisimple flow $\{a_t\}$. The calculation for the upper bound is similar.

\subsection{Main results}
For any $p\in G/\Gamma$, we denote by $\Stab(p)$ the stabilizer of $p$ in $G$. If $p=g\Gamma$, then $\Stab(p)=g\Gamma g^{-1}$ is a lattice conjugate to $\Gamma$. Fix a norm $\|\cdot\|_\mathfrak g$ on the Lie algebra $\mathfrak g$ of $G$, and denote by $d_G(\cdot,\cdot)$ and $d_{G/\Gamma}(\cdot,\cdot)$ the induced distances on $G$ and $G/\Gamma$ respectively. In $G/\Gamma$, we will always discuss Hausdorff dimension with respect to $d_{G/\Gamma}(\cdot,\cdot)$. 

The Cartan decomposition of $G$ with respect to a Cartan involution $\theta$ is written by $$\mathfrak g=\mathfrak k\oplus\mathfrak p$$ where $\mathfrak k$ and $\mathfrak p$ are the $1$-eigenspace and $(-1)$-eigenspace of $\theta$ respectively. Let $\mathfrak a$ be the Lie algebra of the one parameter subgroup $\{a_t\}$. Since $G$ is of rank one, we may assume that $\mathfrak a\subset\mathfrak p$ is a maximal abelian subalgebra of $\mathfrak p$. Let $K$ be the maximal compact subgroup with the Lie algebra $\mathfrak k$. We write the root space decompostion of $\mathfrak g$ with respect to the adjoint action of $\{a_t\}$ as $$\mathfrak g=\mathfrak g_{-2\alpha}\oplus\mathfrak g_{-\alpha}\oplus\mathfrak g_0\oplus\mathfrak g_\alpha\oplus\mathfrak g_{2\alpha}.$$ Here $\alpha$ is a simple root and we think of it as a positive number via the identification $$\mathfrak a^*\cong\mathbb R.$$ In other words, we fix a parametrization of the one-parameter subgroup $\{a_t\}$, and let $\alpha>0$ such that $$\Ad a_t(v)=e^{\alpha t}v\quad(\forall v\in\mathfrak g_\alpha),\qquad\Ad a_t(v)=e^{2\alpha t}v\quad(\forall v\in\mathfrak g_{2\alpha})$$ $$\Ad a_t(v)=e^{-\alpha t}v\quad(\forall v\in\mathfrak g_{-\alpha}),\qquad\Ad a_t(v)=e^{-2\alpha t}v\quad(\forall v\in\mathfrak g_{-2\alpha}).$$ Note that the root spaces $\mathfrak g_{-2\alpha}$ and $\mathfrak g_{2\alpha}$ may be empty. 

\begin{definition}
For any $p\in G/\Gamma$, we define the injectivity radius at $p$ by $$\eta(p)=\inf_{v\in\Stab(p)\setminus\{e\}}d_G(v,e).$$
\end{definition}
\begin{remark}\label{rm}
It is well known that a sequence of points $\{p_n\}\in G/\Gamma$ diverges if and only if $\eta(p_n)\to0$ (\cite{R} Theorem 1.12).
\end{remark}

\begin{definition}\label{def12}
A point $p\in G/\Gamma$ is Diophantine of type $\gamma$ (with respect to $\{a_t\}$) if there exists a constant $C>0$ such that $$\eta(a_tp)\geq Ce^{-\gamma t}\text{ for all } t>0.$$ We denote by $S_\gamma$ the subset of Diophantine points of type $\gamma$, and by $S_\gamma^c$ the complement of $S_\gamma$ in $G/\Gamma$.
\end{definition}
\begin{remark}
Here $\eta(a_tp)$ measures the excursion rate of the orbit $\{a_tp\}$ in $G/\Gamma$. By \cite{D2} and \cite{KM}, the subset $S_0$ has full Hausdorff dimension.
\end{remark}

Now we can state the main theorem in this paper. 
\begin{theorem}[Main theorem]\label{cmaintheorem}
Let $G$ be a linear simple Lie group of real rank one and $\Gamma$ a non-uniform lattice in $G$. Let $U$ be an open subset in $G/\Gamma$. If $\mathfrak g_{2\alpha}=0$, then the Hausdorff dimension of $S_\gamma^c\cap U$ $(0\leq\gamma<\alpha)$ is equal to $$\dim\mathfrak g_{-\alpha}+\dim\mathfrak g_0+\frac{\alpha-\gamma}\alpha\dim\mathfrak g_\alpha.$$
If $\mathfrak g_{2\alpha}\neq0$, then the Hausdorff dimension of $S_\gamma^c\cap U$ $(0\leq\gamma<2\alpha)$ is
$$\dim\mathfrak g_{-2\alpha}+\dim\mathfrak g_{-\alpha}+\dim\mathfrak g_0+\frac{4\alpha-\gamma}{4\alpha}\dim\mathfrak g_\alpha+\frac{2\alpha-\gamma}{2\alpha}\dim\mathfrak g_{2\alpha}.$$
\end{theorem}
\begin{remark}
For $G=\SO(n,1)$, we have $\mathfrak g_{2\alpha}=0$ and Theorem \ref{cmaintheorem} is just a reformulation of Theorem 1 in \cite{MP}. It is also interesting to know what is the Hausdorff dimension of $S_\gamma^c\cap U$ for a semisimple flow on any homogeneous space. 
\end{remark}

Let $\xi_1,\dots,\xi_k$ be the inequivalent cusps of $G/\Gamma$. We fix sufficiently small neighborhoods $Y_i$ of $\xi_i$ in $G/\Gamma$ $(1\leq i\leq k)$ such that these $Y_i$'s are pairwise disjoint. We denote by $\chi_{Y_i}$ the characteristic function of $Y_i$.

\begin{definition}\label{def13}
A point $p\in G/\Gamma$ is Diophantine of type $(\gamma_1,\dots,\gamma_k)$ (with respect to $\{a_t\}$) if there exists a constant $C>0$ such that for any $i\in\{1,2,\dots,k\}$ and any $t>0$, we have $$\eta(a_tp)\chi_{Y_i}(a_tp)\geq Ce^{-\gamma_i t}\chi_{Y_i}(a_tp).$$ We denote by $S_{\gamma_1,\dots,\gamma_k}$ the subset of Diophantine points of type $(\gamma_1,\dots,\gamma_k)$, and by $S_{\gamma_1,\dots,\gamma_k}^c$ the complement of $S_{\gamma_1,\dots,\gamma_k}$ in $G/\Gamma$.
\end{definition}
\begin{remark}
This definition measures different excursion rates $\gamma_i$ of the orbit $\{a_tp\}$ near the cusps $\xi_i$ $(1\leq i\leq k)$.
\end{remark}

The following theorem is a refined version of Theorem \ref{cmaintheorem}. 

\begin{theorem}\label{ccor}
Let $G$, $\Gamma$ and $U$ be as in Theorem \ref{cmaintheorem}. If $\mathfrak g_{2\alpha}=0$, then the Hausdorff dimension of $S_{\gamma_1,\dots,\gamma_k}^c\cap U$ $(0\leq\gamma_i<\alpha,1\leq i\leq k)$ is equal to $$\dim\mathfrak g_{-\alpha}+\dim\mathfrak g_0+\frac{\alpha-\min_{1\leq i\leq k}\gamma_i}\alpha\dim\mathfrak g_\alpha.$$
If $\mathfrak g_{2\alpha}\neq0$, then the Hausdorff dimension of $S_{\gamma_1,\dots,\gamma_k}^c\cap U$ $(0\leq\gamma_i<2\alpha,1\leq i\leq k)$ is
\begin{align*}
\dim\mathfrak g_{-2\alpha}+&\dim\mathfrak g_{-\alpha}+\dim\mathfrak g_0\\
&+\frac{4\alpha-\min_{1\leq i\leq k}\gamma_i}{4\alpha}\dim\mathfrak g_\alpha+\frac{2\alpha-\min_{1\leq i\leq k}\gamma_i}{2\alpha}\dim\mathfrak g_{2\alpha}.
\end{align*}
\end{theorem}
\begin{remark}
 Theorem \ref{ccor} is a generalized version of Theorem 7.1 in \cite{Z}.
\end{remark}

We also deduce a Jarn\'ik-Besicovitch theorem on Diophantine approximation in Heisenberg groups, which follows from a reduced form of Theorem \ref{cmaintheorem} (see Theorem \ref{ctheorem} below). We will work in the setting of \cite{HP3,LV}.

The real Heisenberg group $\mathcal H_{2n-1}(\mathbb R)$ is the subset in $\mathbb C^n$ $$\mathcal H_{2n-1}(\mathbb R):=\{(\zeta,v)\in\mathbb C^{n-1}\times\mathbb C: 2\Re(v)=|\zeta|^2\}$$ with the group multiplication given by $$(\zeta,v)(\zeta',v')=\left(\zeta+\zeta',v+v'+\overline{\zeta}\cdot\zeta'\right).$$ We write $d_{\mathcal H_{2n-1}(\mathbb R)}$ for the left-invariant distance on $\mathcal H_{2n-1}(\mathbb R)$. The subgroups of $\mathbb Q$-points and integer points in $\mathcal H_{2n-1}(\mathbb R)$ are defined respectively as $$\mathcal H_{2n-1}(\mathbb Q)=\mathcal H_{2n-1}(\mathbb R)\cap\mathbb Q[i]^n,\quad\mathcal H_{2n-1}(\mathbb Z)=H_{2n-1}(\mathbb R)\cap\mathbb Z[i]^{n}.$$

 Define $|\cdot|:\mathcal H_{2n-1}(\mathbb R)\to\mathbb R$ by $$|(\zeta,v)|=|v|^\frac12$$ and the Cygan distance on $\mathcal H_{2n-1}(\mathbb R)$ is then given by $$d_{\text{Cyg}}((\zeta,v),(\zeta'v'))=|(\zeta,v)^{-1}(\zeta',v')|.$$
Note that this distance is invariant under the left multilplication of $\mathcal H_{2n-1}(\mathbb R)$.

One can write any rational point $r\in\mathcal H_{2n-1}(\mathbb Q)$ as $$r=p/q,$$ where $p\in\mathcal H_{2n-1}(\mathbb Z)$ and $q\in\mathbb Z[i]$. The height $h(r)$ of $r\in\mathcal H_{2n-1}(\mathbb Q)$ is then defined to be the minimal norm of $q$ among all the possible choices of $r=p/q$.

\begin{definition}\label{def14}
A point $\lambda\in\mathcal H_{2n-1}(\mathbb R)$ is Diophantine of type $\gamma$ if there exists a constant $C>0$ such that $$d_{\text{Cyg}}(\lambda,r)\geq \frac C{(h(r))^\gamma}$$ for any $r\in\mathcal H_{2n-1}(\mathbb Q)$. We will denote by $L_\gamma$ the subset of all Diophantine points of type $\gamma$ in $\mathcal H_{2n-1}(\mathbb R)$ and by $L_\gamma^c$ the complement of $L_\gamma$ in $\mathcal H_{2n-1}(\mathbb R)$.
\end{definition}
\begin{remark}
By Theorem 3.4 in \cite{HP3}, we have $\dim_H L_\gamma^c=2n-1$ $(\gamma\leq1)$.
\end{remark}

\begin{theorem}\label{chptheorem}
The Hausdorff dimension of $L_\gamma^c$ $(\gamma\geq1)$ with respect to the left-invariant distance $d_{\mathcal H_{2n-1}(\mathbb R)}$ is equal to $\frac{\gamma+1}\gamma n-1.$
\end{theorem}
\begin{remark}
Hersonsky and Paulin \cite{HP3,HP4} give a Khintchine-Sullivan type theorem on Diophantine Approximation in Heisenberg groups with respect to the Cygan distance $d_{\text{Cyg}}$ and its induced Cygan measure $\mu_{\text{Cyg}}$. Here Theorem \ref{chptheorem} can be thought of as a Jarn\'ik-Besicovitch theorem in this setting, but with respect to the distance $d_{\mathcal H_{2n-1}(\mathbb R)}$.
\end{remark}

The paper is organized as follows. In section \ref{pre}, we list some notations and prerequisites in this paper, and then in section \ref{re}, we reduce Theorem \ref{cmaintheorem} to Theorem \ref{ctheorem}. In section \ref{le}, some auxiliary results in Lie groups are proved, which are important for our analysis in later sections. In section \ref{crp}, we give a definition of rational points in $G/\Gamma$ and define the denominator of a rational point. With the help of the mixing property of $\{a_t\}$, we will be able to count the rational points with the denominators lying in a range of large scale. This counting result will be used to calculate the Hausdorff dimension of a Cantor type subset. In section \ref{dp}, we study the meaning of a point in $G/\Gamma$ being Diophantine, and show that a Diphantine point could be approximated by rational points, which is similar to the Diophantine approximation in $\mathbb R$. The proofs of Theorem \ref{ctheorem}, Theorem \ref{ccor} and Theorem \ref{chptheorem} will be given in section \ref{pf1}, \ref{pf2} and \ref{pf3}.\\ 

\noindent\textbf{Acknowledgements.} I would like to express my gratitude to my advisor Nimish Shah for his encouragement and invaluable help during my PhD study. Many discussions with him contribute significantly to this paper. I am also grateful to Dmitry Kleinbock and Uri Shapira for insightful comments on related topics, and thank the reviewers whose suggestions improve the paper considerably. The author acknowledges the support of ISF grants number 662/15 and 871/17, and the support at the Technion by a Fine Fellowship. This work has also received funding from the European Research Council (ERC) under the European Union's Horizon 2020 research and innovation programme (grant agreement No. 754475).

\section{Notation and prerequisites}\label{pre}
We say that $B_1\ll B_2$ if there exists a constant $C>0$ such that $B_1\leq CB_2$. If $B_1\ll B_2$ and $B_2\ll B_1$, then we write $B_1\sim B_2$. The implicit constants will be specified in the context.

Let $\text{exp}$ be the exponential map from $\mathfrak g$ to $G$. For any Lie subgroup $H\subseteq G$, denote by $\Lie(H)$ the Lie algebra of $H$ and $\mu_H$ the Haar measure on $H$.

 We write $\Ad$ for the adjoint representation of $G$ on $\mathfrak g$. For convenience, we will also use $\Ad$ for the conjugation of an element $g$ on $G$ $$\Ad(g)x=gxg^{-1},\quad x\in G.$$ In other words, the meaning of $\Ad(g)x$ will depend on $x$: if $x\in\mathfrak g$, then we treat $\Ad$ as the adjoint action of $G$ on $\mathfrak g$; if $x\in G$, then we think of $\Ad$ as the conjugation. As usual, $\ad$ will denote the adjoint representation of the Lie algebra $\mathfrak g$ on $\mathfrak g$.

Let $$\mathfrak n_+=\mathfrak g_\alpha\oplus\mathfrak g_{2\alpha},\quad\mathfrak n_-=\mathfrak g_{-\alpha}\oplus\mathfrak g_{-2\alpha}$$ and let $N_+, N_-$ be the corresponding unipotent subgroups. The exponential map restricted on $\mathfrak n_-$ $$\exp:\mathfrak n_-\to N_-$$ is a diffeomorphism, and for convenience, we will denote its inverse by $$\log: N_-\to\mathfrak n_-.$$ Write $A=\{a_t\}$ and $$A_{s_1,s_2}=\{a_t\in A: s_1\leq t\leq s_2\}.$$ We fix a basis $\{f_1,f_2.\dots\}$ in the vector space $\mathfrak g_\alpha$ and write $$\mathfrak B_{\mathfrak g_\alpha}(r)=\left\{\sum_i x_if_i\in\mathfrak g_\alpha: |x_i|<r\right\}$$ for the open cube around 0 of side length $2r$ in $\mathfrak g_\alpha$. Similarly we could define the open cube $\mathfrak B_{\mathfrak g_{2\alpha}}(r)$ centered at 0 of side length $2r$ in $\mathfrak g_{2\alpha}$.

If $\mathfrak g_{2\alpha}=0$, then $\mathfrak n_+=\mathfrak g_{\alpha}$ and $\mathfrak n_-=\mathfrak g_{-\alpha}$. We will denote by $$B_{N_+}(r)=\exp(\mathfrak B_{\mathfrak g_\alpha}(r))$$ the open cube centered at $e$ with side length $2r$ in $N_+$. If $\mathfrak g_{2\alpha}\neq0$, then $\mathfrak n_+=\mathfrak g_\alpha\oplus\mathfrak g_{2\alpha}$ and denote by $$B_{N_+}(r_1,r_2)=\exp(\mathfrak B_{\mathfrak g_\alpha}(r_1)+\mathfrak B_{\mathfrak g_{2\alpha}}(r_2))$$ for the open box centered at $e$ with side length $2r_1$ in $\mathfrak g_\alpha$-direction and $2r_2$ in $\mathfrak g_{2\alpha}$-direction. A subset in $N_+$ is called an open box if it is a right translate of $B_{N_+}(r)$ for some $r>0$ ($\mathfrak g_{2\alpha}=0$), or $B_{N_+}(r_1,r_2)$ for some $r_1,r_2>0$ ($\mathfrak g_{2\alpha}\neq0$).

The Bruhat decomposition in the real rank one case has the following simple form $$G=MAN_-\cup MAN_-\omega MAN_-$$ where $M$ is the centralizer of $\mathfrak a$ in $K$, $\omega$ is a representative of the non-trivial element in the Weyl group and $N_-=\exp(\mathfrak n_-)$.

We will need the following theorem about estimating Hausdorff dimensions. Let $X$ be a Riemannian manifold, $m$ a volume form and $E$ a compact subset of $X$. Let $\text{diam}(S)$ denote the diameter of a set $S$. A countable collection $\mathcal A$ of compact subsets of $E$ is said to be tree-like \cite{KM} if $\mathcal A$ is the union of finite subcollections $\mathcal A_j$ such that 
\begin{enumerate}
\item $\mathcal A_0=\{E\}$.
\item For any $j$ and $S_1,S_2\in\mathcal A_j$, either $S_1=S_2$ or $S_1\cap S_2=\emptyset$.
\item For any $j$ and $S_1\in\mathcal A_{j+1}$, there exists $S_2\in\mathcal A_j$ such that $S_1\subset S_2$.
\item $d_j(\mathcal A):=\sup_{S\in\mathcal A_j}\text{diam}(S)\to 0$ as $j\to\infty$.
\end{enumerate}

We write $\mathbf A_j=\bigcup_{A\in\mathcal A_j}A$ and define $\mathbf A_\infty=\bigcap_{j\in\mathbb N}\mathbf A_j.$ Moreover, we define $$\Delta_j(\mathcal A)=\inf_{S\in\mathcal A_j}\frac{m(\mathbf A_{j+1}\cap S)}{m(S)}.$$

\begin{theorem}[\cite{M, U}]\label{cthm1}
Let $(X,m)$ be a Riemannian manifold where $m$ is the volume form on $X$. Then for any tree-like collection $\mathcal A$ of subsets of $E$
$$\dim_H(\mathbf A_\infty)\geq \dim X-\limsup_{j\to\infty}\frac{\sum_{i=0}^j\log(\Delta_i(\mathcal A))}{\log(d_{j+1}(\mathcal A))}.$$
\end{theorem}

Note that the definition of tree-like subsets is due to \cite{KM} and here the formulation of Theorem \ref{cthm1} is also borrowed from \cite{KM}, while Theorem \ref{cthm1} is basically proved in the earlier papers \cite{M,U}.

The following theorem describes the fundamental domain of a non-uniform lattice in $G$.  We write the Siegel set $$\Omega(s,V)=KA_{s,\infty}V$$ for some $s\in\mathbb R$ and some compact subset $V\subset N_-$. 

\begin{theorem}[\cite{GR}]\label{cthm2}
There exist $s_0>0$, a compact subset $V_0$ of $N$ and a finite subset $\Sigma$ of $G$ such that the following assertions hold:
\begin{enumerate}
\item $G=\Omega(s_0,V_0)\Sigma\Gamma$.
\item For all $\sigma\in\Sigma$, $\Gamma\cap\sigma^{-1}N_-\sigma$ is a cocompact lattice in $\sigma^{-1}N_-\sigma$.
\item For all compact subsets $V$ of $N_-$ the set $$\{\gamma\in\Gamma|\Omega(s_0,V)\Sigma\gamma\cap\Omega(s_0,V)\neq\emptyset\}$$ is finite.
\item Given a compact subset $V$ of $N_-$ containing $V_0$, there exists $s_1\in(0,s_0)$ such that whenever $\sigma,\tau\in\Sigma$ are such that $\Omega(s_0,V)\sigma\gamma\cap\Omega(s_1,V)\tau$ is non-empty for some $\gamma$ then $\sigma=\tau$ and $\sigma\gamma\sigma^{-1}\in (K\cap Z)\cdot N_-\subset P$.
\end{enumerate}
Here $Z$ is the centralizer of $A=\{a_t\}$ and $P=ZN_-$.
\end{theorem}
\begin{remark}\label{r21}
The formulation of Theorem \ref{cthm2} here is borrowed from \cite{SGD}. The subset $\Sigma$ corresponds to the cusp set $\{\xi_1,\dots,\xi_k\}$ in Definition \ref{def13}.
\end{remark}

\section{Reductions}\label{re}
By the property of Hausdorff dimension, we may assume that $U=Wx\Gamma$ is a sufficiently small neighborhood of a point $x\Gamma\in G/\Gamma$ in Theorem \ref{cmaintheorem}, where $x\in G$ and $W$ is a sufficiently small neighborhood of $e$ in $G$. Furthermore, for any element $g$ in $W$, we can write $$g=n_-amn_+$$ for some $n_-\in N_-$, $a\in A$, $m\in M$ and $n_+\in N_+$. By definition, $gx\Gamma\in S_\gamma$ if and only if $n_+x\Gamma\in S_\gamma$. Hence to prove Theorem \ref{cmaintheorem}, it is enough to prove that for any small open ball $U_0$ at $e$ in $N_+$ we have $$\dim_H(S_\gamma^c\cap U_0(x\Gamma))=\left\{\begin{array}{cc} \frac{\alpha-\gamma}\alpha\dim\mathfrak g_\alpha & \text{if }\mathfrak g_{2\alpha}=0\\ \frac{4\alpha-\gamma}{4\alpha}\dim\mathfrak g_\alpha+\frac{2\alpha-\gamma}{2\alpha}\dim\mathfrak g_{2\alpha} & \text{if }\mathfrak g_{2\alpha}\neq0\end{array}\right..$$ Replacing the lattice $\Gamma$ by $x\Gamma x^{-1}$, we can assume without loss of generality that $x\Gamma=e\Gamma$, and then it suffices to prove the following

\begin{theorem}\label{ctheorem}
Let $U_0$ be a small open ball at $e$ in $N_+$. Then we have $$\dim_H(S_\gamma^c\cap U_0(e\Gamma))=\left\{\begin{array}{cc} \frac{\alpha-\gamma}\alpha\dim\mathfrak g_\alpha & \text{if }\mathfrak g_{2\alpha}=0\\ \frac{4\alpha-\gamma}{4\alpha}\dim\mathfrak g_\alpha+\frac{2\alpha-\gamma}{2\alpha}\dim\mathfrak g_{2\alpha} & \text{if }\mathfrak g_{2\alpha}\neq0\end{array}\right.$$
where $0\leq\gamma<\alpha$ if $\mathfrak g_{2\alpha}=0$ and $0\leq\gamma<2\alpha$ if $\mathfrak g_{2\alpha}\neq0.$
\end{theorem}

Note that for any $\gamma>0$ we have $$S_\gamma^c\cap U_0(e\Gamma)\subset S_0^c\cap U_0(e\Gamma)\subset U_0(e\Gamma),$$ and hence $$\dim_H(S_\gamma^c\cap U_0(e\Gamma))\leq\dim_H(S_0^c\cap U_0(e\Gamma))\leq\dim_H N_+.$$ Because of this, it is enough to consider $\gamma>0$. So we would assume $\gamma>0$ unless otherwise specified.

In the rest of the paper, we will fix the open ball $U_0\subset N_+$ in Theorem \ref{ctheorem}, and study Diophantine points in the space $U_0(e\Gamma)$ instead of $G/\Gamma$. Since $U_0$ is isomorphic to $U_0(e\Gamma)$, we can still write $\mu_{N_+}$ for the $N_+$-invariant measure on $U_0(e\Gamma)$, i.e. $$\mu_{N_+}(B(e\Gamma))=\mu_{N_+}(B)$$ for any Borel subset $B\subset U_0\subset N_+$, and we will use the notations in $U_0$ and $U_0(e\Gamma)$ interchangeably.

\section{Some auxiliary results in Lie groups}\label{le}
In this section, we will prove some Lie group results which will be used later in this note.

\begin{lemma}\label{cp40}
Let $g\in G$ such that $\Ad g(N_-)\cap N_-\neq\{e\}$. Then $g\in MAN_-$. 
\end{lemma}
\begin{proof}
Suppose that $g\notin MAN_-$. By the Bruhat decomposition, $$g\in MAN_-\omega MAN_-.$$ Let $g=g_1\omega g_2$ for some $g_1,g_2\in MAN_-$. Since $MAN_-$ is contained in the normalizer of $N_-$, one could compute
\begin{align*}
\Ad g(N_-)\cap N_-=&\Ad g_1(\Ad(\omega g_2) N_-\cap N_-)\\
=&\Ad g_1(\Ad\omega(N_-)\cap N_-)\\
=&\Ad g_1(N_+\cap N_-)=\{e\},
\end{align*}
which contradicts the assumption that $\Ad g(N_-)\cap N_-\neq\{e\}$.
\end{proof}

\begin{lemma}\label{cp41}
For any $u\in\mathfrak n_-\setminus\{0\}$ and $v\in\mathfrak n_+\setminus\{0\}$ we have $$\|[u,v]\|_{\mathfrak g}\sim\|u\|_{\mathfrak g}\|v\|_{\mathfrak g}.$$ If $u\in\mathfrak g_{-\alpha}$ and $v\in\mathfrak g_{\alpha}$, then $$\|[v,[v,u]]\|_{\mathfrak g}\sim\|v\|_{\mathfrak g}^2\|u\|_{\mathfrak g}.$$ Here the implicit constants depend only on $G$.
\end{lemma}
\begin{proof}
For the first claim, it suffices to prove that for any $u\in\mathfrak n_-\setminus\{0\}$ and $v\in\mathfrak n_+\setminus\{0\}$ $$[u,v]\neq0.$$ Suppose on the contrary that $[u,v]=0$ for some $u\in\mathfrak n_-\setminus\{0\}$ and $v\in\mathfrak n_+\setminus\{0\}$. Then $\Ad(\exp(v))u=u$ and $$\Ad(\exp(v))N_-\cap N_-\neq\{e\},$$ which contradicts Lemma \ref{cp40}.

For the second claim, it is enough to show that for $u\in\mathfrak g_{-\alpha}\setminus\{0\}$ and $v\in\mathfrak g_\alpha\setminus\{0\}$ $$[v,[v,u]]\neq 0.$$ By Lemma 7.73 (b) in \cite{K}, we know that $\ad(v)^2:\mathfrak g_{-\alpha}\to\mathfrak g_{\alpha}$ is surjective, and hence bijective. So $$[v,[v,u]]=\ad(v)^2(u)\neq0.$$ This completes the proof of the lemma.
\end{proof}

\begin{lemma}\label{cp42}
Let $u\in G$ be a unipotent element. Then there exists a unique element $n$ in $N_+\cup\{w\}$ such that $$\Ad n(u)\in N_-.$$ Moreover, if $u\notin N_+$, then this $n\in N_+$.
\end{lemma}
\begin{proof}
We know that there is an element $g\in G$ such that $$\Ad g(u)\in N_-.$$ By the Bruhat decomposition, $g$ is either $ma\bar n\omega$ or $ma\bar nn$ for some $m\in M, a\in A,n\in N_+$ and $\bar n\in N_-$. Since $ma\bar n$ normalizes $N_-$, we have either $$\Ad\omega(u)\in N_-\textup{ or }\Ad n(u)\in N_-.$$ This proves the existence.

Suppose that there are two elements $n_1,n_2\in N_+\cup\{w\}$ such that $$\Ad n_i(u)\in N_-,\; i=1,2.$$ Then $$\Ad(n_2n_1^{-1})N_-\cap N_-\neq\{e\}.$$ By Lemma \ref{cp40}, this implies that $n_2n_1^{-1}\in MAN_-$. Hence by the Bruhat decomposition, $n_2n_1^{-1}=e$ and $n_1=n_2$. This proves the uniqueness.

 The second part of the lemma follows immediately from the first.
\end{proof}

\begin{lemma}\label{nl41}
Suppose that $\mathfrak g_{-2\alpha}\neq0$. Let $g,u\in G$ such that $$\Ad g(u)\in\exp(\mathfrak g_{-2\alpha}).$$ Then there exists an element $n$ in $N_+\cup\{w\}$ such that $$\Ad n(u)\in\exp(\mathfrak g_{-2\alpha}).$$
\end{lemma}
\begin{proof}
The proof is similar to Lemma \ref{cp42}. By the Bruhat decomposition, $g$ is either $ma\bar n\omega$ or $ma\bar nn$ for some $m\in M, a\in A,n\in N_+$ and $\bar n\in N_-$. Since $ma\bar n$ normalizes $\exp(\mathfrak g_{-2\alpha})$, we conclude that either $$\Ad\omega(u)\in\exp(\mathfrak g_{-2\alpha})\textup{ or }\Ad n(u)\in\exp(\mathfrak g_{2\alpha})$$ for some $n\in N_+$.
\end{proof}

\begin{lemma}\label{cp43}
Let $\Sigma$ be as in Theorem \ref{cthm2}. Suppose that $\mathfrak g_{-2\alpha}\neq0$. Then for any $\sigma\in\Sigma$, $\sigma\Gamma\sigma^{-1}\cap\exp(\mathfrak g_{-2\alpha})$ is a lattice in $\exp(\mathfrak g_{-2\alpha})$.
\end{lemma}
\begin{proof}
Let $u\in\mathfrak g_{-\alpha}\setminus\{0\}$. By Lemma 7.73 (a) in \cite{K}, we know that the map $$\ad(u):\mathfrak g_{-\alpha}\to\mathfrak g_{-2\alpha}$$ is surjective. Hence we have $$[\mathfrak g_{-\alpha},\mathfrak g_{-\alpha}]=\mathfrak g_{-2\alpha}.$$ This implies that $$[N_-,N_-]=\exp(\mathfrak g_{-2\alpha}).$$ On the other hand, since $\sigma\Gamma\sigma^{-1}\cap N_-$ is a lattice in $N_-$, by Corollary 1 of Theorem 2.3 in \cite{R}, we know that $\sigma\Gamma\sigma^{-1}\cap[N_-,N_-]$ is a lattice in $[N_-,N_-]$. This completes the proof of the lemma.
\end{proof}

\begin{lemma}\label{nl42}
Suppose $\mathfrak g_{-2\alpha}\neq 0$. Let $\Sigma$ be as in Theorem \ref{cthm2}. Then there exists $C>0$ depending only on $G$ and $\Gamma$ such that for any $\sigma\in\Sigma$ and any $n\in(\sigma\Gamma\sigma^{-1}\cap N_-)\setminus\exp(\mathfrak g_{-2\alpha})$ we have $$\|n_1\|_{\mathfrak g}\geq C$$ where $n=\exp(n_1+n_2)$ for $n_1\in\mathfrak g_{-\alpha}$ and $n_2\in\mathfrak g_{-2\alpha}$.
\end{lemma}
\begin{proof}
Let $\varpi$ denote the natural group homomorphism  $$\varpi:N_-\to N_-/\exp(\mathfrak g_{-2\alpha}).$$ Note that $\exp(\mathfrak g_{-2\alpha})$ is the center of $N_-$.

By Lemma \ref{cp43}, $\sigma\Gamma\sigma^{-1}\cap\exp(\mathfrak g_{-2\alpha})$ is a lattice in $\exp(\mathfrak g_{-2\alpha})$. Hence by Theorem 1.13 in \cite{R}, we get that $(\sigma\Gamma\sigma^{-1}\cap N_-)\exp(\mathfrak g_{-2\alpha})$ is closed in $N_-$, and $\varpi(\sigma\Gamma\sigma^{-1}\cap N_-)$ is a lattice in $N_-/\exp(\mathfrak g_{-2\alpha})$. 

Note that $N_-/\exp(\mathfrak g_{-2\alpha})$ could be identified with the vector space $\mathfrak g_{-\alpha}$, and the map $\varpi$ could be defined as $$\varpi(n)=n_1$$ for any $n=\exp(n_1+n_2)$ with $n_1\in\mathfrak g_{-\alpha}$ and $n_2\in\mathfrak g_{-2\alpha}$. (One could check that this definition of $\varpi$ preserves the group law by Hausdorff-Campbell formula.) Since $\varpi(\sigma\Gamma\sigma^{-1}\cap N_-)$ is a lattice in $N_-/\exp(\mathfrak g_{-2\alpha})\cong\mathfrak g_{-\alpha}$, we conclude that there exits $C>0$ such that for any $n\in(\sigma\Gamma\sigma^{-1}\cap N_-)\setminus\exp(\mathfrak g_{-\alpha})$ $$\|n_1\|_{\mathfrak g}=\|\varpi(n)\|_{\mathfrak g}\geq C.$$ This completes the proof of the lemma.
\end{proof}

\section{Counting rational points}\label{crp}

In this section, we define rational points in $G/\Gamma$ and their denominators. Then we will obtain some information about the distribution of these rational points (Propositions \ref{cp53}, \ref{cp54} and \ref{cp56}), which will be crucial in section \ref{pf1}.

\begin{definition}\label{def51}
A point $p\in G/\Gamma$ is called rational if $\Stab(p)\cap N_-\neq\{e\}$.
\end{definition}

By defintion, if $p$ is rational, then $\eta(a_tp)\to0$ as $t\to\infty$. By Corollary 6.2 in \cite{D0}, we immediately get the following

\begin{proposition}[Corollary 6.2 \cite{D0}]\label{cp51}
$p\in G/\Gamma$ is rational if and only if $p\in\bigcup_{\sigma\in\Sigma}MAN_-\sigma\Gamma$. Here $\Sigma$ is as in Theorem \ref{cthm2}.
\end{proposition}

\begin{definition}
A point $p$ is called $\sigma$-rational for some $\sigma\in\Sigma$ if $p\in MAN_-\sigma\Gamma$.
\end{definition}

\begin{lemma}\label{cp52}
Let $p\in G/\Gamma$ be $\sigma$-rational and suppose that $p=m_1a_1n_1\sigma\Gamma=m_2a_2n_2\sigma\Gamma$. Then $a_1=a_2$.
\end{lemma}
\begin{proof}
Since $m_1a_1n_1\sigma\Gamma=m_2a_2n_2\sigma\Gamma$, the two lattices of $N_-$ below $$(m_1a_1n_1\sigma\Gamma\sigma^{-1}n_1^{-1}a_1^{-1}m_1^{-1})\cap N_-=Ad(m_1a_1n_1)(\sigma\Gamma\sigma^{-1}\cap N_-)$$ $$(m_2a_2n_2\sigma\Gamma\sigma^{-1}n_2^{-1}a_2^{-1}m_2^{-1})\cap N_-=Ad(m_2a_2n_2)(\sigma\Gamma\sigma^{-1}\cap N_-)$$ coincide and hence they have the same co-volume in $N_-$. This implies $a_1=a_2$.
\end{proof}

\begin{definition}\label{def53}
We define the $\sigma$-denominator of a $\sigma$-rational point $p\in G/\Gamma$ by $$d_\sigma(p)=e^{-\alpha t_0}$$ where we write $p=ma_{t_0}n\sigma\Gamma$ for some $t_0\in\mathbb R$.
\end{definition}
\begin{remark}
Note that by Lemma \ref{cp52}, this notion is well-defined.
\end{remark}

\begin{definition}
For any $U\subset N_+$, we will denote by $S_\sigma(U(e\Gamma),l_1,l_2)$ the subset of $\sigma$-rational points  in $U(e\Gamma)$ whose $\sigma$-denominators are between $l_1$ and $l_2$.
\end{definition}

It is known that the action of $\{a_t\}$ on $G/\Gamma$ is mixing: for any $f,g\in L^2(G/\Gamma)$ $$\int_{G/\Gamma} f(a_tx)g(x)mu_{G/\Gamma}(x)\to\int_{G/\Gamma}f(x)d\mu_{G/\Gamma}(x)\int_{G/\Gamma}g(x)d\mu_{G/\Gamma}(x)$$
as $t\to\infty$, where $\mu_{G/\Gamma}$ is the $G$-invariant probability measure on $G/\Gamma$. From this mixing property of $\{a_t\}$, one could deduce (\cite{KM}, Proposition 2.2.1) that for any bounded open subset $U\subset N_+$, $x\in G/\Gamma$ and any compactly supported continuous function $f$ on $G/\Gamma$ $$\frac1{\mu_{N_+}(U)}\int_U f(a_tnx)d\mu_{N_+}(n)\to\int_{G/\Gamma}f(x)d\mu_{G/\Gamma}(x).$$
Now for any bounded open subset $W\subset G/\Gamma$ with $\mu_{G/\Gamma}(\overline W\setminus W)=0$ , by approximating the characteristic function $\chi_W$ of $W$ with compactly supported continuous functions, we have 
\begin{eqnarray}\label{neq0}
\frac1{\mu_{N_+}(U)}\int_U \chi_W(a_tnx)d\mu_{N_+}(n)\to\int_{G/\Gamma}\chi_W(x)d\mu_{G/\Gamma}(x).
\end{eqnarray}
In the following, we will use this formula in the proof of Proposition \ref{cp53} and Proposition \ref{cp54}.

\begin{proposition}\label{cp53}
Suppose that $\mathfrak g_{2\alpha}=0$. Then for any open box $U\subset U_0\subset N_+$, the subset $S_\sigma(U(e\Gamma),C/2,C)$ is finite and $$|S_\sigma(U(e\Gamma),C/2,C)|\sim\mu_{N_+}(U)C^{\dim\mathfrak g_\alpha}$$ for any sufficiently large $C>0$. Here the implicit constant depends only on $G/\Gamma$.
\end{proposition}
\begin{proof}
Since $\mathfrak g_{2\alpha}=0$, we have $\mathfrak n_+=\mathfrak g_{\alpha}$ and $\mathfrak n_-=\mathfrak g_{-\alpha}$. Recall that $B_{N_+}(r)$ denotes the open box centered at $e$ with side length $r$ in $N_+$. 

 Let $n_+\Gamma\in N_+\Gamma$ be a $\sigma$-rational point, and let $m\in M,a_{t_0}\in A$ and $n\in N_-$ such that $$n_+\Gamma=ma_{t_0}n\sigma\Gamma.$$ By Theorem \ref{cthm2}, $N_-\cap\sigma\Gamma\sigma^{-1}$ is a cocompact lattice in $N_-$, so one could assume $n\in \Omega$ for some compact subset $\Omega$ in $N_-$. By Definition \ref{def53}, the $\sigma$-denominator of $n_+\Gamma$ sitting between $C/2$ and $C$ is equivalent to the condition that $$C/2\leq e^{-\alpha t_0}\leq C.$$ This implies that $$a_{\ln C/\alpha}n_+\Gamma\in MA_{-\ln2/\alpha,0}\Omega\sigma\Gamma.$$ This observation will be crucial in our analysis. In the following, we fix a small positive number $\delta>0$.

We first show that 
\begin{align}\label{neqn0}
|S_\sigma(U(e\Gamma),C/2,C)|\ll\mu_{N_+}(U)C^{\dim\mathfrak g_\alpha}
\end{align}
for sufficiently large $C>0$. Fix $\epsilon>0$ such that 
\begin{align}\label{neqn1}
\mu_{N_+}(U)\leq\mu_{N_+}(B_{N_+}(\epsilon)U)\leq2\mu_{N_+}(U).
\end{align}
Suppose that $n_q\Gamma$ is a $\sigma$-rational point in $S_\sigma(U(e\Gamma),C/2,C)$. Then $$a_{\ln C/\alpha}n_q\Gamma\in MA_{-\ln2/\alpha,0}\Omega\sigma\Gamma.$$ This implies that any point $n\Gamma\in B_{N_+}(\delta/C)n_q\Gamma$ satisfies $$a_{\ln C/\alpha}n\Gamma\in B_{N_+}(\delta)MA_{-\ln2/\alpha,0}\Omega\sigma\Gamma.$$ Note that $B_{N_+}(\delta/C)n_q\Gamma\subset B_{N_+}(\epsilon)U(e\Gamma)$ if $C>\delta/\epsilon$. Hence we obtain that for sufficiently large $C>0$
\begin{align}\label{neq1}
&|S_\sigma(U(e\Gamma),C/2,C)|\mu_{N_+}(B_{N_+}(\delta/C))\nonumber\\
\leq&\int_{B_{N_+}(\epsilon)U}\chi_{B_{N_+}(\delta)MA_{-\ln2/\alpha,0}\Omega\sigma\Gamma}(a_{\ln C/\alpha}n\Gamma)d\mu_{N_+}(n)
\end{align}
By formula (\ref{neq0}) and equation (\ref{neqn1}), as $C\to\infty$, we have
\begin{align}\label{ceq2}
 &\int_{B_{N_+}(\epsilon)U}\chi_{B_{N_+}(\delta)MA_{-\ln2/\alpha,0}\Omega\sigma\Gamma}(a_{\ln C/\alpha}n\Gamma)d\mu_{N_+}(n)\nonumber\\
 \sim&\mu_{N_+}(B_{N_+}(\epsilon)U)\int_{G/\Gamma}\chi_{B_{N_+}(\delta)MA_{-\ln2/\alpha,0}\Omega\sigma\Gamma}(x)d\mu_{G/\Gamma}(x)\nonumber\\
 \sim&\mu_{N_+}(U)\delta^{\dim\mathfrak g_\alpha}.
\end{align}
Since $\mu_{N_+}(B_{N_+}(\delta/C))\sim(\delta/C)^{\dim\mathfrak g_\alpha}$, equations (\ref{neq1}) and (\ref{ceq2}) then yield equation (\ref{neqn0}).

Now we show that 
\begin{align}\label{neqn2}
\mu_{N_+}(U)C^{\dim\mathfrak g_\alpha}\ll|S_\sigma(U(e\Gamma),C/2,C)|
\end{align}
for sufficiently large $C>0$. Pick a small box $U'$ in $U$ with 
\begin{align}\label{neqn3}
\mu_{N_+}(U')=\frac12\mu_{N_+}(U).
\end{align}
Suppose that a point $n\Gamma\in U'(e\Gamma)$ satisfies $$a_{\ln C/\alpha}n\Gamma\in B_{N_+}(\delta)MA_{-\ln2/\alpha,0}\Omega\sigma\Gamma.$$ Then there is a point $n_q\Gamma\in B_{N_+}(\delta/C)n\Gamma$ such that $$a_{\ln C/\alpha}n_q\Gamma\in MA_{-\ln2/\alpha,0}\Omega\sigma\Gamma.$$ This means that $n_q\Gamma$ is a $\sigma$-rational point with its $\sigma$-denominator between $C/2$ and $C$, and $n\Gamma\in B_{N_+}(\delta/C)n_q\Gamma$. Note that $$n_q\Gamma\in B_{N_+}(\delta/C)n\Gamma\subset (B_{N_+}(\delta/C)U')(e\Gamma).$$ Hence we obtain that 
\begin{align}\label{neq2}
&\int_{U'}\chi_{B_{N_+}(\delta)MA_{-\ln2/\alpha,0}\Omega\sigma\Gamma}(a_{\ln C/\alpha}n\Gamma)d\mu_{N_+}(n)\nonumber\\
\leq& |S_\sigma((B_{N_+}(\delta/C)U')(e\Gamma),C/2,C)|\mu_{N_+}(B_{N_+}(\delta/C)).
\end{align}
By formula (\ref{neq0}) and equation (\ref{neqn3}), as $C\to\infty$, we have
\begin{align}\label{ceqn2}
 &\int_{U'}\chi_{B_{N_+}(\delta)MA_{-\ln2/\alpha,0}\Omega\sigma\Gamma}(a_{\ln C/\alpha}n\Gamma)d\mu_{N_+}(n)\nonumber\\
 \sim&\mu_{N_+}(U')\int_{G/\Gamma}\chi_{B_{N_+}(\delta)MA_{-\ln2/\alpha,0}\Omega\sigma\Gamma}(x)d\mu_{G/\Gamma}(x)\nonumber\\
 \sim&\mu_{N_+}(U)\delta^{\dim\mathfrak g_\alpha}.
\end{align}
Since $\mu_{N_+}(B_{N_+}(\delta/C))\sim(\delta/C)^{\dim\mathfrak g_\alpha}$ and $B_{N_+}(\delta/C)U'(e\Gamma)\subset U(e\Gamma)$ for sufficiently large $C$, equations (\ref{neq2}) and (\ref{ceqn2}) then imply equation (\ref{neqn2}).
 This completes the proof of the proposition.
\end{proof}

\begin{proposition}\label{cp54}
Suppose that $\mathfrak g_{2\alpha}\neq0$. Then for any open box $U\subset U_0\subset N_+$,  the subset $S_\sigma(U(e\Gamma),C/2,C)$ is finite and $$|S_\sigma(U(e\Gamma),C/2,C)|\sim\mu_{N_+}(U)C^{\dim\mathfrak g_\alpha+2\dim\mathfrak g_{2\alpha}}$$ for any sufficiently large $C>0$. Here the implicit constant depends only on $G/\Gamma$.
\end{proposition}
\begin{proof}
It is almost identical to Proposition \ref{cp53} but computations involved here will be more complicated. Recall that $B_{N_+}(r_1,r_2)$ denotes the open box centered at $e$ with side length $r_1$ in $\mathfrak g_\alpha$-direction and $r_2$ in $\mathfrak g_{2\alpha}$-direction. 

We follow the arguments in Proposition \ref{cp53}. We show that 
\begin{align}\label{neqn4}
|S_\sigma(U(e\Gamma),C/2,C)|\ll\mu_{N_+}(U)C^{\dim\mathfrak g_\alpha+2\dim\mathfrak g_{2\alpha}}.
\end{align}
The proof of the other inequality is similar to Proposition \ref{cp53}.

Let $\delta>0$ be a small positive number. Fix $\epsilon>0$ such that 
\begin{align}\label{neqn5}
\mu_{N_+}(U)\leq\mu_{N_+}(B_{N_+}(\epsilon,\epsilon)U)\leq2\mu_{N_+}(U).
\end{align}
Suppose that $n_q\Gamma$ is a $\sigma$-rational point in $S_\sigma(U(e\Gamma),C/2,C)$. Then $$a_{\ln C/\alpha}n_q\Gamma\in MA_{-\ln2/\alpha,0}\Omega\sigma\Gamma.$$ This implies that any $n\Gamma\in B_{N_+}(\delta/C,\delta/C^2)n_q\Gamma$ satisfies $$a_{\ln C/\alpha}n\Gamma\in B_{N_+}(\delta,\delta)MA_{-\ln2/\alpha,0}\Omega\sigma\Gamma.$$ Note that $B_{N_+}(\delta/C,\delta/C^2)n_q\Gamma\subset B(\epsilon,\epsilon)U(e\Gamma)$ for sufficiently large $C$. Hence we obtain that for sufficiently large $C$
\begin{align}\label{neq5}
&|S_\sigma(U(e\Gamma),C/2,C)|\mu_{N_+}(B_{N_+}(\delta/C,\delta/C^2))\nonumber\\
\leq&\int_{B_{N_+}(\epsilon,\epsilon)U}\chi_{B_{N_+}(\delta,\delta)MA_{-\ln2/\alpha,0}\Omega\sigma\Gamma}(a_{\ln C/\alpha}n\Gamma)d\mu_{N_+}(n).
\end{align}
By formula (\ref{neq0}) and equation (\ref{neqn5}) is asymptotically equal to 
\begin{align}\label{ceq5}
&\int_{B_{N_+}(\epsilon,\epsilon)U}\chi_{B_{N_+}(\delta,\delta)MA_{-\ln2/\alpha,0}\Omega\sigma\Gamma}(a_{\ln C/\alpha}n\Gamma)d\mu_{N_+}(n)\nonumber\\
\sim&\mu_{N_+}(B_{N_+}(\epsilon,\epsilon)U)\int_{G/\Gamma}\chi_{B_{N_+}(\delta,\delta)MA_{-\ln2/\alpha,0}\Omega\sigma\Gamma}(x)d\mu_{G/\Gamma}(x)\nonumber\\
\sim&\mu_{N_+}(U)\delta^{\dim\mathfrak g_\alpha+\dim\mathfrak g_{2\alpha}}.
\end{align}
Since $\mu_{N_+}(B_{N_+}(\delta/C,\delta/C^2))\sim(\delta/C)^{\dim\mathfrak g_\alpha}(\delta/C^2)^{\dim\mathfrak g_{2\alpha}}$, equations (\ref{neq5}) and (\ref{ceq5}) then yield equation (\ref{neqn4}).
\end{proof}

\begin{definition}\label{def55}
We define the denominator of a rational point $p$ by $$d(p)=\inf_{v\in\Stab(p)\cap\exp(\mathfrak g_{-\beta})\setminus\{e\}}\|\log v\|_{\mathfrak g}^{\frac\alpha\beta}$$ where $\beta=\alpha$ if $\mathfrak g_{2\alpha}=0$ and $\beta=2\alpha$ if $\mathfrak g_{2\alpha}\neq0$.
\end{definition}
\begin{remark}
Note that by Lemma \ref{cp43}, $\Stab(p)\cap\exp(\mathfrak g_{-2\alpha})\neq\{e\}$ is a lattice in $\exp(\mathfrak g_{-2\alpha})$ if $\mathfrak g_{2\alpha}\neq0$ and so $d(p)$ is well-defined.
\end{remark}

\begin{definition}
For any $U\subset N_+$, we will denote by $S(U(e\Gamma),l_1,l_2)$ the subset of rational points in $U(e\Gamma)$ whose denominators are between $l_1$ and $l_2$.
\end{definition}

\begin{lemma}\label{cp55}
Let $p\in G/\Gamma$ be a rational point. Then $$d(p)\sim d_\sigma(p)$$ whenever $p$ is a $\sigma$-rational point for some $\sigma\in\Sigma$. Here the implicit constant depends only on $G$ and $\Gamma$.
\end{lemma}

\begin{proof}
Let $p=ma_{t_0}n\sigma\Gamma$ for some $\sigma\in\Sigma$. Suppose that $\mathfrak g_{2\alpha}=0$. Then 
\begin{align*}
\Stab(p)\cap\exp(\mathfrak g_{-\alpha})=&Ad(ma_{t_0}n)(\sigma\Gamma\sigma^{-1})\cap\exp(\mathfrak g_{-\alpha})\\
=&Ad(ma_{t_0}n)(\sigma\Gamma\sigma^{-1}\cap\exp(\mathfrak g_{-\alpha})).
\end{align*}
 Since $m,n,\sigma$ are all in compact subsets in $G$, this implies $$d(p)\sim e^{-\alpha t_0}=d_\sigma(p).$$ The proof in the case $\mathfrak g_{2\alpha}\neq0$ is similar.
\end{proof}

\begin{proposition}\label{cp56}
Let $U\subset U_0\subset N_+$ be an open box in $N_+$. For any sufficiently large $C>0$, the subset $S(U(e\Gamma),C/2,C)$ is finite and
\begin{enumerate}
\item if $\mathfrak g_{2\alpha}=0$, then $$|S(U(e\Gamma),C/2,C)|\sim\mu_{N_+}(U)C^{\dim\mathfrak g_\alpha}.$$
\item if $\mathfrak g_{2\alpha}\neq0$, then $$|S(U(e\Gamma),C/2,C)|\sim\mu_{N_+}(U)C^{\dim\mathfrak g_\alpha+2\dim\mathfrak g_{2\alpha}}.$$
\end{enumerate}
Here the implicit constants depend only on $G$ and $\Gamma$.
\end{proposition}
\begin{proof}
This follows immediately from Proposition \ref{cp51}, Proposition \ref{cp53}, Proposition \ref{cp54} and Lemma \ref{cp55}.
\end{proof}

\section{Diophantine points}\label{dp}
In this section, we study Diophantine points in $G/\Gamma$. We will prove that the Diophantine points could be approximated by rational points (Propositions \ref{cp63}, \ref{cp64}, \ref{cp66}, \ref{cp67}), which could be considered as an analogue of Diophantine approximation in $\mathbb R$. These results will be used to construct tree-like subsets in the proof of Theorem \ref{ctheorem}.

\subsection{Preliminaries}\label{dp1}
Before we dive into the main discussion, we make some preparations in this subsection. Our goal is Proposition \ref{cp61}.

First we recall some properties of Lie groups and Lie algebras. For any $\rho>0$, let $\mathfrak B(\rho)$ denote the open ball of radius $\rho$ around $0$ in $\mathfrak g$ and $B(\rho):=\exp(\mathfrak B(\rho))$. There is a constant $0<\rho_0<1$ such that 
\begin{enumerate}
\item $\exp:\mathfrak B(\rho_0)\to B(\rho_0)$ is a diffeomorphism.
\item $B(\rho_0)$ is a Zassenhaus neighborhood of $e$ in $G$, which is a neighborhood of $e$ such that for any discrete group $\Delta\subset G$, there exists a connected nilpotent subgroup $F\subset G$ such that $\Delta\cap B(\rho_0)\subset F$. Note that any Lie group $G$ admits a Zaussenhaus neighborhood by Theorem 8.16 in \cite{R}.
\item $B(\rho_0)$ is a neighborhood of $e$ satisfying Corollary 11.18 in \cite{R} for $G$ and $\Gamma$.
\end{enumerate}
 There is a constant $\kappa>1$ depending only on $G$ such that 
\begin{enumerate}
\item for any $v\in\mathfrak B(\rho_0)$, one has $$\frac1\kappa d_G(\exp(v),e)\leq\|v\|_{\mathfrak g}\leq\kappa d_G(\exp(v),e).$$
\item if $v=v_-+v_0+v_+\in\mathfrak g$ where $v_-\in\mathfrak n_-$, $v_0\in\mathfrak g_0$ and $v_+\in\mathfrak n_+$, then $$\|v_-\|_{\mathfrak g},\|v_0\|_{\mathfrak g},\|v_+\|_{\mathfrak g}\leq\kappa\|v\|_{\mathfrak g}.$$
\end{enumerate}
In the following, we will fix such $\rho_0$ and $\kappa$. We denote by $\mathbb R_+$ the set of positive numbers.

\begin{lemma}\label{nl61}
Let $p\in G/\Gamma$ be a non-rational point. If $p$ is not Diophantine of type $\gamma$, then there are sequences $\{s_n\}$ and $\{t_n\}$ in $\mathbb R_+$ with $s_n,t_n\to\infty$ such that 
\begin{enumerate}
\item for any $n>0$, $s_n<t_n<s_{n+1}<t_{n+1}$;
\item one has $$\eta(a_{s_n}p)e^{\gamma s_n}\to\infty\text{ as }n\to\infty,\quad\eta(a_{t_n}p)e^{\gamma t_n}=\rho_0,$$ and $$\eta(a_{t}p)e^{\gamma t}>\rho_0,\quad\forall t\in(s_n,t_n).$$
\end{enumerate}
\end{lemma}
\begin{proof}
By Proposition \ref{cp51}, we know that $\{a_tp\}$ is not a divergent orbit as $t\to\infty$. So by Remark \ref{rm} there exists $\delta_1>0$ and a sequence $s_n\to\infty$ such that $$\eta(a_{s_n}p)\geq \delta_1.$$ Let $\delta_2=\min\{\delta_1/2,\rho_0\}$. By Definition \ref{def12}, if $p$ is not Diophantine of type $\gamma$, then there exists a sequence $l_n\to\infty$ such that $$\eta(a_{l_n}p)\leq\frac{\delta_2}2e^{-\gamma l_n}.$$ By passing to subsequences, one could make $\{s_n\}$ and $\{l_n\}$ alternating in the sense that $s_n<l_n<s_{n+1}<l_{n+1}$ for any $n\in\mathbb N$.

Now for any $n>0$, consider the function $$f(t)=\eta(a_{t}p)e^{\gamma t}$$ on $[s_n,l_n]$. We have $f(s_n)\geq \delta_1e^{\gamma s_n}$ and $$f(s_n)\to\infty\text{ as }n\to\infty,$$ while $f(l_n)\leq \delta_2/2<\rho_0$. By intermediate value theorem, the function $f(t)$ attains $\rho_0$ somewhere in $[s_n,l_n]$. Define $$t_n=\inf\{t \in[s_n,l_n]:f(t)=\rho_0\}.$$ Then by continuity of $f(t)$, we have $f(t_n)=\rho_0$, and for any $t\in(s_n,t_n)$ $$f(t)>\rho_0.$$ The sequences $\{s_n\}$ and $\{t_n\}$ then satisfy the desired properties.
\end{proof}

\begin{lemma}\label{nl62}
Let $p\in G/\Gamma,\{s_n\},\{t_n\}$ be as in Lemma \ref{nl61}. For sufficiently large $n\in\mathbb N$, there exists $l_n\in[s_n,t_n]$ and $v_n\in\Stab(a_{l_n}p)$ such that
\begin{enumerate}
\item $v_n$ is unipotent and $$\rho_0e^{-\gamma l_n}\leq\eta(a_{l_n}p)\leq d_G(v_n,e)\leq3\kappa^3\rho_0e^{-\gamma l_n};$$
\item if $v_n=\exp(v_{n,-}+v_{n,0}+v_{n,+})$ where $v_{n,-}\in\mathfrak n_-$, $v_{n,0}\in\mathfrak g_0$ and $v_{n,+}\in\mathfrak n_+$, then $$\frac{\rho_0}{3\kappa^3}e^{-\gamma l_n}\leq\|v_{n,-}\|_{\mathfrak g}\leq\kappa^2\rho_0e^{-\gamma l_n}.$$
\end{enumerate}
\end{lemma}
\begin{proof}
By Lemma \ref{nl61}, we know that $\eta(a_{t_n}p)=\rho_0e^{-\gamma t_n}$, which means that there exists $w_n\in\Stab(a_{t_n}p)$ such that $$d_G(w_n,e)=\eta(a_{t_n}p)=\rho_0e^{-\gamma t_n}.$$ It follows from Corollary 11.18 in \cite{R} that $w_n$ is unipotent for sufficiently large $t_n>0$. Write $$w_n=\exp(w_{n,-}+w_{n,0}+w_{n,+})$$ where $w_{n,-}\in\mathfrak n_-$, $w_{n,0}\in\mathfrak g_0$ and $w_{n,+}\in\mathfrak n_+$. Then $$\|w_{n,-}\|_{\mathfrak g},\|w_{n,0}\|_{\mathfrak g},\|w_{n,+}\|_{\mathfrak g}\leq\kappa^2d_G(w_n,e)=\kappa^2\rho_0e^{-\gamma t_n}.$$ If $\|w_{n,-}\|_{\mathfrak g}\geq\rho_0 e^{-\gamma t_n}$, then we are done by taking $l_n:=t_n$ and $v_n:=w_n$.

Suppose that $$\|w_{n,-}\|_{\mathfrak g}<\rho_0 e^{-\gamma t_n}.$$ Since $\Ad(a_t)$ $(t<0)$ expands $\mathfrak n_-$, we can define $\tau_n>0$ to be the smallest positive number such that 
\begin{align}\label{neq61}
\|\Ad(a_{-\tau_n})w_{n,-}\|_{\mathfrak g}=\rho_0e^{-\gamma t_n}.
\end{align}
 We claim that in this case we can take $$l_n:=t_n-\tau_n\textup{ and } v_n:=\Ad(a_{-\tau_n})w_n\in\Stab(a_{l_n}p).$$ Note that $v_n$ is unipotent.
 
First, we prove that $l_n\in[s_n,t_n]$ for sufficiently large $n$. Indeed, by the definition of $\tau_n$, $$\|\Ad(a_{-\tau})w_{n,-}\|_{\mathfrak g}\leq\rho_0 e^{-\gamma t_n}$$ for any $\tau\in(0,\tau_n)$. Since $\{\Ad a_t\}$ $(t<0)$ stablizes $\mathfrak g_0$ and contracts $\mathfrak n_+$, we also have $$\|\Ad(a_{-\tau})w_{n,0}\|_{\mathfrak g}=\|w_{n,0}\|_{\mathfrak g}\leq\kappa^2\rho_0 e^{-\gamma t_n}$$ $$\|\Ad(a_{-\tau})w_{n,+}\|_{\mathfrak g}\leq\|w_{n,+}\|_{\mathfrak g}\leq\kappa^2\rho_0 e^{-\gamma t_n}$$ for any $\tau\in(0,\tau_n)$. This implies that for any $\tau\in(0,\tau_n)$, $$d_G(\Ad(a_{-\tau})w_n,e)\leq3\kappa^3\rho_0 e^{-\gamma t_n}.$$ Note that $\Ad(a_{-\tau})w_n\in\Stab(a_{t_n-\tau}p)$. By the definition of the injectivity radius function $\eta$, we have 
\begin{align}\label{neq62}
\eta(a_{t_n-\tau}p)\leq d_G(\Ad(a_{-\tau})w_n,e)\leq3\kappa^3\rho_0 e^{-\gamma t_n}
\end{align}
for any $\tau\in(0,\tau_n)$. This implies that $$\eta(a_{t_n-\tau}p)e^{\gamma(t_n-\tau)}\leq 3\kappa^3\rho_0,\quad\tau\in(0,\tau_n).$$ By condition 2 in Lemma \ref{nl61}, we conclude that $$s_n\leq l_n=t_n-\tau_n\leq t_n$$ if $n$ is sufficiently large. 

Now we prove that $$l_n:=t_n-\tau_n\textup{ and } v_n:=\Ad(a_{-\tau_n})w_n\in\Stab(a_{l_n}p)$$ satisfy conditions 1 and 2. Since $l_n\in[s_n,t_n]$, by Lemma \ref{nl61},  we have
\begin{align}\label{neq63}
\eta(a_{l_n}p)=\eta(a_{t_n-\tau_n}p)>\rho_0 e^{-\gamma l_n}=\rho_0 e^{-\gamma(t_n-\tau_n)}.
\end{align} 
Comparing equations (\ref{neq62}) and (\ref{neq63}), we have 
\begin{align}\label{neq64}
e^{\gamma\tau_n}\leq3\kappa^3.
\end{align}
Condition 1 then follows from equations (\ref{neq62}) and (\ref{neq63}), while condition 2 follows from equations (\ref{neq61}) and (\ref{neq64}). This completes the proof of the lemma.
\end{proof}.

\begin{definition}
Let $p\in G/\Gamma$, $\gamma>0$, $t\in\mathbb R_+$ and $v\in\Stab(a_t p)$. We say that $t $ and $v$ satisfy $\gamma$-condition if
\begin{enumerate}
\item $v$ is unipotent and $$\rho_0e^{-\gamma t }\leq\eta(a_t p)\leq d_G(v,e)\leq3\kappa^3\rho_0e^{-\gamma t };$$
\item if $v=\exp(v_-+v_0+v_+)$ where $v_-\in\mathfrak n_-$, $v_0\in\mathfrak g_0$ and $v_+\in\mathfrak n_+$, then $$\frac{\rho_0}{3\kappa^3}e^{-\gamma t }\leq\|v_-\|_{\mathfrak g}\leq\kappa^2\rho_0e^{-\gamma t }.$$
\end{enumerate}
\end{definition}

Summarizing Lemma \ref{nl61} and Lemma \ref{nl62}, we obtain the following proposition.

\begin{proposition}\label{cp61}
Let $p\in G/\Gamma$ be a non-rational point. If $p$ is not Diophantine of type $\gamma$, then there exist a sequence $t_n\to\infty$ and a sequence $v_n\in\Stab(a_{t_n}p)$ such that for each $n\in\mathbb N$, $t_n $ and $v_n$ satisfy $\gamma$-condition.
\end{proposition}

\subsection{Case $\mathfrak g_{2\alpha}=0$}
We first study the Diophantine points in the case $\mathfrak g_{2\alpha}=0$, where the analysis is easy to follow. This case would provide a guide for the case $\mathfrak g_{2\alpha}\neq0$. The goal of this subsection is to prove Proposition \ref{cp63} and Proposition \ref{cp64}.

\begin{proposition}\label{cp62}
 Let $p\in U_0(e\Gamma)\subset G/\Gamma$ be non-rational, $t_0\in\mathbb R_+$ a sufficiently large number and $v_0\in\Stab(a_{t_0}p)$ such that $t_0$ and $v_0$ satisfy $\gamma$-condition. Then there is a rational point $q\in U_0(e\Gamma)$ with $d(q)\sim e^{(\alpha-\gamma)t_0}$ (the implicit constant depends only on $G/\Gamma$) such that $$p\in B_{N_+}(Cd(q)^{-\frac\alpha{\alpha-\gamma}})q$$ for some constant $C>0$ depending only on $G/\Gamma$ and $\gamma$ .
\end{proposition}
\begin{proof}
By $\gamma$-condition, we have $v_0\notin N_+$. Let $v_0=\exp(v)$ for some $v\in\mathfrak g$. By Lemma \ref{cp42}, there is a unique $n\in N_+$ such that $\Ad n(v_0)\in N_-$ and $$\Ad n(v)=z\in\mathfrak n_-.$$ Now set $n^{-1}=\exp(u)$ for some $u\in\mathfrak n_+$. Then we have 
\begin{equation}\label{eqn61}
z+[u,z]+[u,[u,z]]/2=v
\end{equation}
 $$z\in\mathfrak n_-=\mathfrak g_{-\alpha} ,\quad[u,z]\in\mathfrak g_0,\quad [u,[u,z]]\in\mathfrak n_+=\mathfrak g_\alpha .$$ By $\gamma$-condition, $$\|v\|_{\mathfrak g}\sim d_G(v_0,e)\sim e^{-\gamma t_0}$$ and hence 
\begin{align}\label{neq8}
\|[u,z]\|_{\mathfrak g}\ll\|v\|_{\mathfrak g}\sim e^{-\gamma t_0}.
\end{align} 
By equation (\ref{eqn61}) and $\gamma$-condition, we also have $$\|z\|_{\mathfrak g}\sim e^{-\gamma t_0}.$$ By equation (\ref{neq8}) and Lemma \ref{cp41}, this implies $$\|u\|_{\mathfrak g}\leq C_1$$ for some constant $C_1>0$, and hence $u$ and $n$ are bounded.

Now since $\Ad n(v_0)\in N_-$, by definition, we know that $na_{t_0}p$ is rational and so is $a_{-t_0}na_{t_0}p$. Let $$q=(a_{-t_0}na_{t_0})p.$$ By the boundedness of $n$ and $\gamma$-condition, we have $$\eta(na_{t_0}p)\sim \eta(a_{t_0}p)\sim d_G(v_0,e)\sim e^{-\gamma t_0}.$$ Since $\Ad n(v_0)\in\Stab(na_{t_0}p)\cap\exp(\mathfrak g_{-\alpha})$ and $d_G(v_0,e)\sim e^{-\gamma t_0}$, we also have an estimate for the denominator of $na_{t_0}p$ $$e^{-\gamma t_0}\sim\eta(na_{t_0}p)\ll d(na_{t_0}p)\ll d_G(\Ad n(v_0),e)\sim e^{-\gamma t_0}.$$ Hence the denominator of $q=a_{-t_0}na_{t_0}p$ is equal to $$d(q)=e^{\alpha t_0}d(na_{t_0}p)\sim e^{\alpha t_0}\eta(na_{t_0}p)\sim e^{(\alpha-\gamma)t_0}.$$ Therefore we have $$d_G(a_{-t_0}na_{t_0},e)\ll e^{-\alpha t_0}\sim d(q)^{-\frac\alpha{\alpha-\gamma}},$$ and $$p=(a_{-t_0}na_{t_0})^{-1}q\in B_{N_+}(Cd(q)^{-\frac\alpha{\alpha-\gamma}})q$$ for some constant $C>0$ depending only on $G/\Gamma$ and $\gamma$. This completes the proof of the proposition.
\end{proof}

\begin{proposition}\label{cp63}
 Let $p\in U_0(e\Gamma)\subset G/\Gamma$ be a non-rational point. If $p$ is not Diophantine of type $\gamma$, then there exist an constant $C>0$ depending only on $G/\Gamma$ and $\gamma$, and a sequence of distinct rational points $q_n\in U_0(e\Gamma)$ with $d(q_n)\to\infty$ such that $$p\in B_{N_+}(Cd(q_n)^{-\frac\alpha{\alpha-\gamma}})q_n.$$
\end{proposition}
\begin{proof}
By Proposition \ref{cp61}, there exist infinitely many $t_n\to\infty$ and $v_n\in\Stab(a_{t_n}p)$ satisfying $\gamma$-condition. Therefore, by Proposition \ref{cp62}, there exist infinitely many rational points $q_n\in U_0(e\Gamma)$  with $d(q_n)\sim e^{(\alpha-\gamma)t_n}$ such that $$p\in B(Cd(q_n)^{-\frac\alpha{\alpha-\gamma}})q_n$$ for some constant $C>0$.
\end{proof}

\begin{proposition}\label{cp64}
 Let $p\in U_0(e\Gamma)$ and let $\epsilon>0$ be a sufficiently small number. If there exist a constant $C>0$ and a sequence $q_n\in U_0(e\Gamma)$ of distinct rational points with $d(q_n)\to\infty$ such that $$p\in B_{N_+}(C d(q_n)^{-\frac\alpha{\alpha-(\gamma+\epsilon)}})q_n,$$ then $p$ is not Diophantine of type $\gamma$.
\end{proposition}
\begin{proof}
Let $t_n=\ln d(q_n)^{\frac1{\alpha-(\gamma+\epsilon)}}$. Then we have $$d_{G/\Gamma}(a_{t_n}q_n,a_{t_n}p)\leq C$$ and hence $$\eta(a_{t_n}p)\sim\eta(a_{t_n}q_n).$$ Note that $$\eta(a_{t_n}q_n)\ll d(a_{t_n}q_n)=d(q_n)/e^{\alpha t_n}=e^{-(\gamma+\epsilon) t_n}.$$ This implies that $$\eta(a_{t_n}p)\leq C'e^{-(\gamma+\epsilon)t_n}$$ for some constant $C'>0$ and infinitely many $t_n\to\infty$. By definition, $p$ is not Diophantine of type $\gamma$.
\end{proof}

\subsection{Case $\mathfrak g_{2\alpha}\neq0$}
Now we consider the case $\mathfrak g_{2\alpha}\neq0$. We aim to prove Proposition \ref{cp66} and Proposition \ref{cp67}.

\begin{lemma}\label{nl63}
Let $\Sigma$, $K$ and $V_0$ be as in Theorem \ref{cthm2}. Let $\sigma\in\Sigma$, $k\in K$, $n_0\in V_0$ and $s$ a sufficiently large number. Then we have $$\Stab(ka_sn_0\sigma\Gamma)\cap B(\rho_0)\subset\Ad(ka_sn_0)(\sigma\Gamma\sigma^{-1}\cap N_-).$$
\end{lemma}
\begin{proof}
Since $s$ is sufficiently large and $\{\Ad a_t\}$ $(t>0)$ contracts $N_-$, there is a nontrivial element $$v\in\Ad(ka_sn_0)(\sigma\Gamma\sigma^{-1}\cap N_-)\cap B(\rho_0).$$ As $B(\rho_0)$ is a Zassenhauss neighborhood, there exists a connected nilpotent subgroup $F\subset G$ such that $$\Stab(ka_sn_0\sigma\Gamma)\cap B(\rho_0)\subset F.$$ By Corollary 11.18 in \cite{R},  every element in $\Stab(ka_sn_0\sigma\Gamma)\cap B(\rho_0)$ is unipotent. Hence we could assume that $F$ is a unipotent subgroup, and $F$ is contained in a conjugate of $N_-$, say, $$F\subset \Ad(g)N_-.$$ Now as $v\in F$, we have $$v\in\Ad(g)N_-\cap\Ad(ka_sn_0)N_-.$$ By Lemma 3.4 in \cite{BZ}, this implies that $$\Ad(g)N_-=\Ad(ka_sn_0)N_-.$$ Therefore, we have 
\begin{align*}
\Stab(ka_sn_0\sigma\Gamma)\cap B(\rho_0)\subset&\Stab(ka_sn_0\sigma\Gamma)\cap F\\
\subset&\Stab(ka_sn_0\sigma\Gamma)\cap\Ad(ka_sn_0)N_-\\
=&\Ad(ka_sn_0)(\sigma\Gamma\sigma^{-1}\cap N_-).
\end{align*}
This completes the proof of the lemma.
\end{proof}

\begin{lemma}\label{nl64}
Let $\Sigma$, $K$ and $V_0$ be as in Theorem \ref{cthm2}. Let $\sigma\in\Sigma$, $k\in K$, $n_0\in V_0$. Then there exists $C>0$ depending only on $\Sigma$, $K$, $V_0$ and $G/\Gamma$ such that $$\frac{d_G(\Ad(ka_sn_0)u,e)}{d_G(\Ad(ka_sn_0)v,e)}\geq\frac{Ce^{\alpha s}}{\|\log v\|_{\mathfrak g}}$$ for any nontrivial elements $$u\in(\sigma\Gamma\sigma^{-1}\cap N_-)\setminus\exp(\mathfrak g_{-2\alpha})\textup{ and }v\in\sigma\Gamma\sigma^{-1}\cap\exp(\mathfrak g_{-2\alpha})$$ with $$\Ad(ka_sn_0)u,\Ad(ka_sn_0)v\in B(\rho_0).$$
\end{lemma}
\begin{proof}
Write $$u=\exp(u_1+u_2)\textup{ and }v=\exp(v_2)$$ where $u_1\in\mathfrak g_{-\alpha}$, $u_2\in\mathfrak g_{-2\alpha}$ and $v_2\in\mathfrak g_{-2\alpha}$. Since $\Sigma$, $K$ and $V_0$ are compact, we compute
\begin{align*}
\frac{d_G(\Ad(ka_sn_0)u,e)}{d_G(\Ad(ka_sn_0)v,e)}\geq&\frac1{\kappa^2}\frac{\|\Ad(ka_sn_0)(u_1+u_2)\|_{\mathfrak g}}{\|\Ad(ka_sn_0)v_2\|_{\mathfrak g}}\\
\geq&C_1\frac{\|\Ad(a_sn_0)(u_1+u_2)\|_{\mathfrak g}}{\|\Ad(a_sn_0)v_2\|_{\mathfrak g}}\\
\geq&C_2\frac{\|\Ad(a_s)u_1\|_{\mathfrak g}}{\|\Ad(a_s)v_2\|_{\mathfrak g}}=C_2e^{\alpha s}\frac{\|u_1\|_{\mathfrak g}}{\|v_2\|_{\mathfrak g}}\\
\geq&\frac{Ce^{\alpha s}}{\|v_2\|_{\mathfrak g}}.
\end{align*}
for some $C_1,C_2,C>0$. Note that the last step follows from Lemma \ref{nl42}. 
\end{proof}

\begin{lemma}\label{cp}
Suppose that $\mathfrak g_{2\alpha}\neq0$. Let $p\in U_0(e\Gamma)\subset G/\Gamma$ be non-rational, $t_0\in\mathbb R_+$ a sufficiently large number and $v_0\in\Stab(a_{t_0}p)$ such that $t_0$ and $v_0$ satisfy $\gamma$-condition. Then there exists $n\in N_+$ such that $$\Ad n(v_0)\in\exp(\mathfrak g_{-2\alpha}).$$
\end{lemma}
\begin{proof}
Since $t_0$ and $v_0$ satisfy $\gamma$-condition, one has 
\begin{align}\label{neq}
d_G(v_0,e)\sim\eta(a_{t_0}p),\quad\eta(a_{t_0}p)\sim e^{-\gamma t_0}.
\end{align}
As $t_0\in\mathbb R_+$ is sufficiently large, by Remark \ref{rm}, $a_{t_0}p$ is near one of the cusps of $G/\Gamma$. In view of this, and by Theorem \ref{cthm2}, we could write $$a_{t_0}p=ka_sn_0\sigma\Gamma$$ for some $k\in K$, $a_s\in A$, $n_0\in V_0$, $\sigma\in\Sigma$ where $s$ is a sufficiently large number. By Lemma \ref{nl63}, we have 
$$v_0\in\Stab(a_{t_0}p)\cap B(\rho_0)\subset\Ad(ka_sn_0)(\sigma\Gamma\sigma^{-1}\cap N_-).$$
By Lemma \ref{nl64} and equation (\ref{neq}), this implies that $v_0$ must sit inside $$\Ad(ka_sn_0)(\sigma\Gamma\sigma^{-1}\cap\exp(\mathfrak g_{-2\alpha})).$$ In other words, we can find $n\in G$ such that 
\begin{equation}\label{ceq}
\text{Ad}n(v_0)\in\exp(\mathfrak g_{-2\alpha}).
\end{equation}
 By Lemma \ref{nl41}, we can assume that $n\in N_+\cup\{w\}$. By $\gamma$-condition, we have $v_0\notin N_+$. Hence by equation (\ref{ceq}), we get $n\in N_+$. This completes the proof of the lemma.
\end{proof}

\begin{proposition}\label{cp65}
 Let $p\in U_0(e\Gamma)\subset G/\Gamma$ be non-rational, $t_0\in\mathbb R_+$ a sufficiently large number and $v_0\in\Stab(a_{t_0}p)$ such that $t_0$ and $v_0$ satisfy $\gamma$-condition. Then there is a rational point $q\in U_0(e\Gamma)$ with $d(q)\sim e^{(2\alpha-\gamma){t_0}/2}$ (the implicit constant depends only on $G/\Gamma$) such that $$p\in B_{N_+}(Cd(q)^{-\frac{2\alpha}{2\alpha-\gamma}}, Cd(q)^{-\frac{4\alpha}{2\alpha-\gamma}})q$$ for some constant $C>0$ depending only on $G/\Gamma$ and $\gamma$.
\end{proposition}
\begin{proof}
By Lemma \ref{cp}, we know that there exists $n\in N_+$ such that $$\Ad(n)(v_0)\in\exp(\mathfrak g_{-2\alpha}).$$ Let $v_0=\exp(v)$ for some $v\in\mathfrak g$.  Then we have $$\Ad n(v)=z\in\mathfrak g_{-2\alpha}.$$ Now set $n^{-1}=\exp(u)$ for some $u=u_1+u_2$ with $u_1\in\mathfrak g_\alpha$ and $u_2\in\mathfrak g_{2\alpha}$. Then we have 
\begin{align}\label{eqn62}
z&+[u_1,z]+([u_2,z]+[u_1,[u_1,z]]/2)\nonumber\\
&+([u_1,[u_1,[u_1,z]]]/6+[u_1,[u_2,z]]/2+[u_2,[u_1,z]]/2)+\dots=v
\end{align}
where $$a=[u_1,z]\in\mathfrak g_{-\alpha},b=[u_2,z]+[u_1,[u_1,z]]/2\in\mathfrak g_0$$ 
$$c=[u_1,[u_1,[u_1,z]]]/6+[u_1,[u_2,z]]/2+[u_2,[u_1,z]]/2\in\mathfrak g_\alpha.$$ We also have the following identities for $c$
\begin{align}
c&=[u_1,[u_1,[u_1,z]]]/6+[u_2,[u_1,z]]\nonumber\\
&=[u_1,[u_1,[u_1,z]]]/6+[u_2,a]\label{deq}\\
&=[u_1,[u_1,[u_1,z]]]/6+[u_1,[u_2,z]]\nonumber\\
&=-[u_1,[u_1,[u_1,z]]]/3+[u_1,b]\nonumber\\
&=-[u_1,[u_1,a]]/3+[u_1,b].\label{deq1}
\end{align}
Since $\|v\|_{\mathfrak g}\sim d_G(v_0,e)\sim e^{-\gamma t_0}$, we have $$\|z\|_{\mathfrak g},\|a\|_{\mathfrak g},\|b\|_{\mathfrak g},\|c\|_{\mathfrak g}\ll e^{-\gamma t_0}.$$ By equation (\ref{eqn62}) and $\gamma$-condition, we have $\|z+a\|_{\mathfrak g}\sim e^{-\gamma t_0}$, and hence $$\text{either }\|z\|_{\mathfrak g}\sim e^{-\gamma t_0}\text{ or }\|a\|_{\mathfrak g}\sim e^{-\gamma t_0}.$$ 

If $\|z\|_{\mathfrak g}\sim e^{-\gamma t_0}$, then by Lemma \ref{cp41}, we have $$\|u_1\|_{\mathfrak g}\|z\|_{\mathfrak g}\ll\|a\|_{\mathfrak g}\ll e^{-\gamma t_0},\quad\|u_1\|_{\mathfrak g}\leq C_1$$ for some constant $C_1>0$, and $$\|u_2\|_{\mathfrak g}\|z\|_{\mathfrak g}\ll\|b\|_{\mathfrak g}+\|u_1\|_{\mathfrak g}^2\|z\|_{\mathfrak g}\ll e^{-\gamma t_0},\quad\|u_2\|_{\mathfrak g}\leq C_2$$ for some constant $C_2>0$. 

If $\|a\|_{\mathfrak g}\sim e^{-\gamma t_0}$, then by Lemma \ref{cp41} $$\|a\|_{\mathfrak g}\ll\|u_1\|_{\mathfrak g}\|z\|_{\mathfrak g}\ll\|u_1\|_{\mathfrak g}e^{-\gamma t_0},\quad\|u_1\|_{\mathfrak g}\geq C_3$$ for some constant $C_3>0$. By equation (\ref{deq1}) and Lemma \ref{cp41} $$\|u_1\|_{\mathfrak g}^2\|a\|_{\mathfrak g}\ll\|u_1\|_{\mathfrak g}\|b\|_{\mathfrak g}+\|c\|_{\mathfrak g},\quad \|u_1\|_{\mathfrak g}\ll1+\frac1{\|u_1\|_{\mathfrak g}}\leq C_4$$ for some constant $C_4>0$. By equation (\ref{deq}) and Lemma \ref{cp41} $$\|u_2\|_{\mathfrak g}\|a\|_{\mathfrak g}\ll\|c\|_{\mathfrak g}+ \|u_1\|_{\mathfrak g}^3\|z\|_{\mathfrak g},\quad\|u_2\|_{\mathfrak g}\leq C_5$$ for some constant $C_5>0$. Either case, we have that $u=u_1+u_2$ and $n$ are bounded.

Now since $\Ad(n)(v_0)\in\exp(\mathfrak g_{-2\alpha})$, by definition, $na_{t_0}p$ is rational and so is $a_{-t_0}na_{t_0}p$. Let $$q=(a_{-t_0}na_{t_0})p.$$ By the boundedness of $n$ and $\gamma$-condition, we know that $$\eta(na_{t_0}p)\sim\eta(a_{t_0}p)\sim e^{-\gamma t_0}.$$ Since $\Ad n(v_0)\in\Stab(na_{t_0}p)\cap\exp(\mathfrak g_{-2\alpha})$ and $d_G(v_0,e)\sim e^{-\gamma t_0}$, we also have an estimate for the denominator of $na_{t_0}p$ $$e^{-\gamma t_0}\sim\eta(na_{t_0}p)\ll d(na_{t_0}p)^2\ll d_G(\Ad n(v_0),e)\sim e^{-\gamma t_0}.$$ Hence the denominator of $q=a_{-t_0}na_{t_0}p$ is equal to  $$d(q)=e^{\alpha t_0}d(na_{t_0}p)\sim e^{\alpha t_0}\eta(na_{t_0}p)^\frac12\sim e^{(2\alpha-\gamma)t_0/2}.$$ This implies that $$a_{-t_0}na_{t_0}\in B_{N_+}(C_6e^{-\alpha t_0}, C_6e^{-2\alpha t_0})\subset B_{N_+}(Cd(q)^{-\frac{2\alpha}{2\alpha-\gamma}}, Cd(q)^{-\frac{4\alpha}{2\alpha-\gamma}})$$ and $$p=(a_{-t_0}na_{t_0})^{-1}q\in B_{N_+}(Cd(q)^{-\frac{2\alpha}{2\alpha-\gamma}}, Cd(q)^{-\frac{4\alpha}{2\alpha-\gamma}})q$$ for some $C_6, C>0$. This completes the proof of the proposition.
\end{proof}

\begin{proposition}\label{cp66}
 Let $p\in U_0(e\Gamma)$ be a non-rational point. If $p$ is not Diophantine of type $\gamma$, then there exist a constant $C>0$ depending only on $G/\Gamma$ and $\gamma$, and a sequence $q_n\in U(e\Gamma)$ of distinct rational points with $d(q_n)\to\infty$ such that $$p\in B_{N_+}(Cd(q_n)^{-\frac{2\alpha}{2\alpha-\gamma}}, Cd(q_n)^{-\frac{4\alpha}{2\alpha-\gamma}})q_n.$$
\end{proposition}
\begin{proof}
It is similar to Proposition \ref{cp63}.
\end{proof}

\begin{proposition}\label{cp67}
 Let $p\in U_0(e\Gamma)$ and let $\epsilon>0$ be a sufficiently small number. If there exist a constant $C>0$ and a sequence $q_n\in U_0(e\Gamma)$ of distinct rational points with $d(q_n)\to\infty$ such that $$p\in B_{N_+}(Cd(q_n)^{-\frac{2\alpha}{2\alpha-(\gamma+\epsilon)}}, Cd(q_n)^{-\frac{4\alpha}{2\alpha-(\gamma+\epsilon)}})q_n,$$ then $p$ is not Diophantine of type $\gamma$.
\end{proposition}
\begin{proof}
Let $t_n=\ln d(q_n)^{\frac2{2\alpha-(\gamma+\epsilon)}}$. Then the proof is identical to that of Proposition \ref{cp64}.
\end{proof}

\section{Proof of Theorem \ref{ctheorem}}\label{pf1}
In this section, we prove Theorem \ref{ctheorem}.

\begin{lemma}\label{nl7}
Let $B(\rho_0)$ be the neighborhood of $e$ in $G$ defined as in subsection \ref{dp1}. Suppose that there are rational points $q_1, q_2$ in $U_0(e\Gamma)$ and $n\in N_+$ such that $nq_1=q_2$ and $$d(q_1),d(q_2)\leq\rho_0/2,\;\Ad n(B(\rho_0/2))\subset B(\rho_0).$$ Then $n=e$ and $q_1=q_2$.
\end{lemma}
\begin{proof}
Since $B(\rho_0)$ is a Zassenhaus neighborhood and $d(q_1),d(q_2)\leq\rho_0/2$, there are connected nilpotent subgroups $F_1,F_2\subset G$ such that $$\{e\}\neq\Stab(q_i)\cap B(\rho_0)\subset F_i,\; i=1,2.$$ By Corollary 11.18 in \cite{R},  every element in $\Stab(q_i)\cap B(\rho_0)$ $(i=1,2)$ is unipotent. Hence we could assume that each $F_i$ $(i=1,2)$ is a unipotent subgroup, and contained in a conjugate of $N_-$. By enlarging $F_i$, we could further assume that each $F_i$ is a conjugate of $N_-$ $$F_i=\Ad(g_i)N_-,\; i=1,2.$$

Now since $$\Ad(n)(\Stab(q_1)\cap B(\rho_0/2))\subset \Stab(q_2)\cap B(\rho_0),$$ we have 
\begin{align}\label{eqn0}
\Ad n(F_1)\cap F_2\neq \{e\}.
\end{align}
As $d(q_1), d(q_2)\leq\rho_0/2$, there are nontrivial elements $$w_i\in\Stab(q_i)\cap N_-\cap B(\rho_0),\;i=1,2.$$ This implies that $$w_i\in N_-\cap F_i=N_-\cap\Ad(g_i)N_-,\; i=1,2.$$ By Lemma 3.4 in \cite{BZ}, we have $$F_1=F_2=N_-.$$ By equation (\ref{eqn0}), we get $\Ad(n)N_-\cap N_-\neq\{e\}$. Since $n\in N_+$, by Lemma \ref{cp40}, we conclude that $n=e$.
\end{proof}

\begin{proposition}\label{cp71}
Suppose $\mathfrak g_{2\alpha}=0$ and let $0<\gamma<\alpha$. There exist $r_0$ and $C_0>0$ such that for a rational point $q\in U_0(e\Gamma)$ and for any sufficiently large $l>0$, the open box $$B_{N_+}(r_0d(q)^{-\frac\alpha{\alpha-\gamma}})q$$ contains at least $$C_0l^{\dim\mathfrak g_\alpha}\mu_{N_+}(B_{N_+}(r_0d(q)^{-\frac\alpha{\alpha-\gamma}})q)$$ disjoint open boxes of the form $B_{N_+}(r_0d(\tilde q)^{-\frac\alpha{\alpha-\gamma}})\tilde q$ where $\tilde q$'s are rational points with denominators between $l$ and $2l$.
\end{proposition}
\begin{proof}
We fix a sufficiently small $r_0>0$. By Proposition \ref{cp56}, for the open box $$B_{N_+}(r_0d(q)^{-\frac\alpha{\alpha-\gamma}})q,$$ there exists a large constant $L_q>0$ such that for any $l>L_q$ there are at least $$C_0l^{\dim\mathfrak g_\alpha}\mu_{N_+}(B_{N_+}(r_0d(q)^{-\frac\alpha{\alpha-\gamma}})q)$$ rational points in $B_{N_+}(r_0d(q)^{-\frac\alpha{\alpha-\gamma}})q$ with denominators between $l$ and $2l$, where the constant $C_0>0$ depending only on $G/\Gamma$. For each such rational point $\tilde q$, we construct an open box $$B_{N_+}(r_0d(\tilde q)^{-\frac\alpha{\alpha-\gamma}})\tilde q$$ around $\tilde q$. To prove the proposition, we only need to show that these open boxes are disjoint. 

Let $q_1$ and $q_2$ be two rational points in $B_{N_+}(r_0d(q)^{-\frac\alpha{\alpha-\gamma}})q$ with denominators lying between $l$ and $2l$. Suppose that $$B_{N_+}(r_0d(q_1)^{-\frac\alpha{\alpha-\gamma}})q_1\textup{ and }B_{N_+}(r_0d(q_2)^{-\frac\alpha{\alpha-\gamma}})q_2$$ are not disjoint, where $l\leq d(q_1),d(q_2)\leq2l$. Then there exists an element $n\in N_+$ such that $$nq_1=q_2,\quad n\in B_{N_+}(r_0d(q_1)^{-\frac\alpha{\alpha-\gamma}}+r_0d(q_2)^{-\frac\alpha{\alpha-\gamma}}).$$ By applying $a_{t_0}$ with $t_0=\ln (l/r_0)/\alpha$ on the equation $nq_1=q_2$, we have 
\begin{equation}\label{ceq0}
(a_{t_0} na_{-t_0})a_{t_0}q_1=a_{t_0}q_2
\end{equation}
and by calculations $$d(a_{t_0}q_1)\sim d(a_{t_0}q_2)\sim r_0\text{ and }d_G(a_{t_0} na_{-t_0},e)\ll l^{-\frac\gamma{\alpha-\gamma}}.$$ Since $r_0$ is sufficiently small and $l$ is sufficiently large, by Lemma \ref{nl7} and equation (\ref{ceq0}), we have $q_1=q_2$. This completes the proof of the proposition.
\end{proof}

\begin{proposition}\label{cp72}
Suppose $\mathfrak g_{2\alpha}\neq0$ and $0<\gamma<2\alpha$. There exist $r_0$ and $C_0>0$ such that for an open box $U(e\Gamma)\subset U_0(e\Gamma)$ and for any sufficiently large $l>0$ the open box $U(e\Gamma)$ contains at least $$C_0l^{\dim\mathfrak g_\alpha+2\dim\mathfrak g_{2\alpha}}\mu_{N_+}(U)$$ disjoint open boxes of the form $B_{N_+}(r_0d(q)^{-\frac{2\alpha}{2\alpha-\gamma}}, r_0d(q)^{-\frac{4\alpha}{2\alpha-\gamma}})q$ where $q$'s are rational points with denominators between $l$ and $2l$.
\end{proposition}
\begin{proof}
We fix a sufficiently small $r_0>0$. By Proposition \ref{cp56}, for any $U(e\Gamma)$, there exists a large constant $L>0$ such that $l>L$ there are at least $$C_0l^{\dim\mathfrak g_\alpha+2\dim\mathfrak g_{2\alpha}}\mu_{N_+}(U)$$ rational points in $U(e\Gamma)$ with denominators between $l$ and $2l$ for some constant $C_0>0$ depending only on $G/\Gamma$. For each such rational point $q$, we construct an open box $$B_{N_+}(r_0d(q)^{-\frac{2\alpha}{2\alpha-\gamma}}, r_0d(q)^{-\frac{4\alpha}{2\alpha-\gamma}})q$$ around $q$. To prove the proposition, we only need to show that these open boxes are disjoint. 

Let $q_1$ and $q_2$ be two rational points in $U(e\Gamma)$ with denominators lying between $l$ and $2l$. Suppose that $$B_{N_+}(r_0d(q_1)^{-\frac{2\alpha}{2\alpha-\gamma}}, r_0d(q_1)^{-\frac{4\alpha}{2\alpha-\gamma}})q_1\textup{ and }B_{N_+}(r_0d(q_2)^{-\frac{2\alpha}{2\alpha-\gamma}}, r_0d(q_2)^{-\frac{4\alpha}{2\alpha-\gamma}})q_2$$ are not disjoint, where $l\leq d(q_1),d(q_2)\leq2l$. Then there exists an element $n\in N_+$ such that $nq_1=q_2$ and 
\begin{align*}
n\in B_{N_+}(r_0d(q_1)^{-\frac{2\alpha}{2\alpha-\gamma}}&+r_0d(q_2)^{-\frac{2\alpha}{2\alpha-\gamma}},\\
&r_0d(q_1)^{-\frac{4\alpha}{2\alpha-\gamma}}+r_0d(q_2)^{-\frac{4\alpha}{2\alpha-\gamma}}+r_0^2(d(q_1)d(q_2))^{-\frac{2\alpha}{2\alpha-\gamma}}).
\end{align*}
 By applying $a_{t_0}$ with $t_0=\ln (l/r_0^\frac12)/\alpha$ on the equation $nq_1=q_2$, we have 
\begin{equation}\label{ceqn}
(a_{t_0} na_{-t_0})a_{t_0}q_1=a_{t_0}q_2
\end{equation}
and by calculations $$d(a_{t_0}q_1)\sim d(a_{t_0}q_2)\sim r_0^\frac12\text{ and }d_G(a_{t_0} na_{-t_0},e)\ll l^{-\frac\gamma{2\alpha-\gamma}}.$$ 
The rest of the proof is similar to Proposition \ref{cp71}.
\end{proof}

Now we are ready to prove Theorem \ref{ctheorem}.

\begin{proof}[Proof of Theorem \ref{ctheorem} in the case $\mathfrak g_{2\alpha}=0$.]
Let $C>0$ be as in Proposition \ref{cp63}. For any $l>0$, by Proposition \ref{cp63}, the following is an open cover of $S_\gamma^c\cap U_0(e\Gamma)$ in $U_0(e\Gamma)$: $$\{B_{N_+}(Cd(q)^{-\frac{\alpha}{\alpha-\gamma}})q: q\text{ rational in $U_0(e\Gamma)$ with }d(q)>l\}.$$
We denote this open cover by $\mathcal Y(l)$. Note that the diameters of open subsets in $\mathcal Y(l)$ uniformly converge to 0 as $l\to\infty$. Now let $\delta>0$ and $l$ sufficiently large. By Proposition \ref{cp56}, for $\mathcal Y(l)$ we calculate 
\begin{eqnarray*}
&{}&\sum_{q\in U_0(e\Gamma), d(q)>l}\text{diam}^\delta(B_{N_+}(Cd(q)^{-\frac\alpha{\alpha-\gamma}})q)\\
&\leq&C^{-\delta\frac\alpha{\alpha-\gamma}}\sum_{n\in\mathbb N}\sum_{2^n\leq d(q)\leq 2^{n+1}} d(q)^{-\delta\frac\alpha{\alpha-\gamma}}\\
&\sim&\sum_{n\in\mathbb N}(2^n)^{\dim\mathfrak g_\alpha}(2^n)^{-\delta\frac\alpha{\alpha-\gamma}},
\end{eqnarray*}
which converges if $\delta>\frac{\alpha-\gamma}\alpha\dim\mathfrak g_\alpha$.
This implies that $$\dim_H S_\gamma^c\cap U_0(e\Gamma)\leq\frac{\alpha-\gamma}\alpha\dim\mathfrak g_\alpha.$$

For the lower bound, we fix a sufficiently small $\epsilon>0$ and construct a tree-like subset in $U_0(e\Gamma)$ by induction. Let $\mathcal A_0=\{U_0(e\Gamma)\}$ and $\mathbf A_0=U_0(e\Gamma)$. Let $r_0$ be as in Proposition \ref{cp71}, and pick a sufficiently large number $l_1$. Define $$\mathcal A_1=\left\{B_{N_+}(r_0d(q)^{-\frac\alpha{\alpha-(\gamma+\epsilon)}})q\Bigg|q\in S(U,l_1/2,l_1)\right\}$$ and $\mathbf A_1=\bigcup\mathcal A_1$. Suppose that we find $l_1<l_2<\dots<l_j$ and construct families $\mathcal A_j,\mathcal A_{j-1},\dots,\mathcal A_0$ and subsets $\mathbf A_j\subseteq\mathbf A_{j-1}\subseteq\dots\subseteq \mathbf A_1\subseteq \mathbf A_0$. Now by Proposition \ref{cp71}, we can find a sufficiently large $l_{j+1}>0$ such that 
\begin{enumerate}
\item $\log l_{j+1}\geq j^2\log(l_jl_{j-1}\dots l_1)$.
\item For every $B_{N_+}(r_0d(q)^{-\frac\alpha{\alpha-(\gamma+\epsilon)}})q\in\mathcal A_j$, it contains at least $$C_0l_{j+1}^{\dim\mathfrak g_\alpha}\mu_{N_+}(B_{N_+}(r_0d(q)^{-\frac\alpha{\alpha-(\gamma+\epsilon)}})q)$$ open boxes of the form $B_{N_+}(r_0d(\tilde q)^{-\frac\alpha{\alpha-(\gamma+\epsilon)}})\tilde q$ with $$\tilde q\in S(B_{N_+}(r_0d(q)^{-\frac\alpha{\alpha-(\gamma+\epsilon)}})q,l_{j+1}/2,l_{j+1}).$$
\end{enumerate}
We denote by $\mathcal A_{j+1}$ the collection of the boxes $B_{N_+}(r_0d(\tilde q)^{-\frac\alpha{\alpha-(\gamma+\epsilon)}})\tilde q$ with $$\tilde q\in S(B_{N_+}(r_0d(q)^{-\frac\alpha{\alpha-(\gamma+\epsilon)}})q,l_{j+1}/2,l_{j+1}),$$ where $B_{N_+}(r_0d(q)^{-\frac\alpha{\alpha-(\gamma+\epsilon)}})q$ runs through all the open boxes in $\mathcal A_j$. Then we define $\mathbf A_{j+1}=\bigcup\mathcal A_{j+1}$.

Now we take $\mathbf A_\infty=\bigcap_{j=0}^\infty\mathbf A_j$ and $\mathcal A=\bigcup_{j=0}^\infty\mathcal A_j$. By the construction of $\mathbf A_j$'s and Proposition \ref{cp64}, we know that $\mathbf A_\infty\subset S_\gamma^c\cap U_0(e\Gamma)$. Also we have $$\Delta_j(\mathcal A)\sim l_{j+1}^{\dim\mathfrak g_\alpha} l_{j+1}^{-\frac\alpha{\alpha-(\gamma+\epsilon)}\dim\mathfrak g_\alpha}\text{ and }d_j(\mathcal A)=r_0 l_j^{-\frac\alpha{\alpha-(\gamma+\epsilon)}}.$$ By Theorem \ref{cthm1}, we compute 
\begin{align*}
\dim_H(\mathbf A_\infty)&\geq\dim\mathfrak g_\alpha-\limsup_{j\to\infty}\frac{\sum_{i=0}^j\log\left(l_{i+1}^{-\frac{\gamma+\epsilon}{\alpha-(\gamma+\epsilon)}\dim\mathfrak g_\alpha}\right)}{\log\left(l_{j+1}^{-\frac\alpha{\alpha-(\gamma+\epsilon)}}\right)}\\
&=\dim\mathfrak g_\alpha\left(1-\frac{\gamma+\epsilon}\alpha\right).
\end{align*}
Let $\epsilon\to0$ and we get $$\dim_H S_\gamma^c\cap U_0(e\Gamma)\geq\frac{\alpha-\gamma}\alpha\dim\mathfrak g_\alpha.$$ This finishes the proof of the theorem in the case $\mathfrak g_{2\alpha}=0$.
\end{proof}

\begin{proof}[Proof of Theorem \ref{ctheorem} in the case $\mathfrak g_{2\alpha}\neq 0$.]
Let $C>0$ be as in Proposition \ref{cp66}. For any $l>0$, by Proposition \ref{cp66}, the following is an open cover of $S_\gamma^c\cap U_0(e\Gamma)$ in $U_0(e\Gamma)$: $$\{B_{N_+}(Cd(q)^{-\frac{2\alpha}{2\alpha-\gamma}}, Cd(q)^{-\frac{4\alpha}{2\alpha-\gamma}})q: q\text{ rational in $U_0(e\Gamma)$ with }d(q)>l\}.$$
We denote this open cover by $\mathcal Y(l)$. 

For each $B_{N_+}(Cd(q)^{-\frac{2\alpha}{2\alpha-\gamma}}, Cd(q)^{-\frac{4\alpha}{2\alpha-\gamma}})q\in\mathcal Y(l)$, we divide it into cubes of side length $Cd(q)^{-\frac{4\alpha}{2\alpha-\gamma}}$, and there are $(d(q)^{\frac{2\alpha}{2\alpha-\gamma}})^{\dim\mathfrak g_\alpha}$ such cubes. Let $\mathcal F(l)$ be the collection of all these cubes as $B_{N_+}(Cd(q)^{-\frac{2\alpha}{2\alpha-\gamma}}, Cd(q)^{-\frac{4\alpha}{2\alpha-\gamma}})q$ runs through all open boxes in $\mathcal Y(l)$. Note that $\mathcal F(l)$ is an open cover of $S_\gamma^c\cap U_0(e\Gamma)$, and the diameters of the cubes in $\mathcal F(l)$ uniformly converge to 0 as $l\to\infty$. 

Now let $\delta>0$ and $l$ sufficiently large. By Proposition \ref{cp56}, for $\mathcal F(l)$ we calculate
\begin{eqnarray*}
\sum_{B\in\mathcal F(l)}\text{diam}(B)^\delta&\sim&\sum_{d(q)>l} d(q)^{-\delta\frac{4\alpha}{2\alpha-\gamma}}(d(q)^{\frac{2\alpha}{2\alpha-\gamma}})^{\dim\mathfrak g_\alpha}\\
&\leq&\sum_{n\in\mathbb N}\sum_{2^n\leq d(q)\leq 2^{n+1}} d(q)^{-\delta\frac{4\alpha}{2\alpha-\gamma}}(d(q)^{\frac{2\alpha}{2\alpha-\gamma}})^{\dim\mathfrak g_\alpha}\\
&\sim&\sum_{n\in\mathbb N}(2^n)^{\dim\mathfrak g_\alpha+2\dim\mathfrak g_{2\alpha}}(2^n)^{-\delta\frac{4\alpha}{2\alpha-\gamma}}((2^n)^{\frac{2\alpha}{2\alpha-\gamma}})^{\dim\mathfrak g_\alpha},\end{eqnarray*}
which converges if $\delta>\frac{4\alpha-\gamma}{4\alpha}\dim\mathfrak g_\alpha+\frac{2\alpha-\gamma}{2\alpha}\dim\mathfrak g_{2\alpha}$.
This implies that $$\dim_H S_\gamma^c\cap U_0(e\Gamma)\leq\frac{4\alpha-\gamma}{4\alpha}\dim\mathfrak g_\alpha+\frac{2\alpha-\gamma}{2\alpha}\dim\mathfrak g_{2\alpha}.$$ 

For the lower bound, we fix a sufficiently small $\epsilon>0$ and construct a tree-like set in $U_0$ by induction. Let $\mathcal A_0=\{U_0(e\Gamma)\}$ and $\mathbf A_0=U_0(e\Gamma)$. Let $r_0$ be as in Proposition \ref{cp72} and pick a sufficiently large number $l_1$. Define $$\mathcal A_1'=\left\{B_{N_+}(r_0d(q)^{-\frac{2\alpha}{2\alpha-(\gamma+\epsilon)}}, r_0d(q)^{-\frac{4\alpha}{2\alpha-(\gamma+\epsilon)}})q\Bigg|q\in S(U_0(e\Gamma),l_1/2,l_1)\right\}.$$ For each $B_{N_+}(r_0d(q)^{-\frac{2\alpha}{2\alpha-(\gamma+\epsilon)}}, r_0d(q)^{-\frac{4\alpha}{2\alpha-(\gamma+\epsilon)}})q$ in $\mathcal A_1'$, we devide it into cubes of side length $r_0d(q)^{-\frac{4\alpha}{2\alpha-(\gamma+\epsilon)}}$. Let $\mathcal A_1$ be the family of all these cubes as $B_{N_+}(r_0d(q)^{-\frac{2\alpha}{2\alpha-(\gamma+\epsilon)}}, r_0d(q)^{-\frac{4\alpha}{2\alpha-(\gamma+\epsilon)}})q$ runs through $\mathcal A_1'$, and define $\mathbf A_1=\bigcup\mathcal A_1$. Suppose that we find $l_1<l_2<\dots<l_j$, and construct families $\mathcal A_j,\mathcal A_{j-1},\dots,\mathcal A_0$ and subsets $\mathbf A_j\subseteq\mathbf A_{j-1}\subseteq\dots\subseteq \mathbf A_1\subseteq \mathbf A_0$. Now by Proposition \ref{cp72}, we can find a sufficiently large $l_{j+1}>0$ such that 
\begin{enumerate}
\item $\log l_{j+1}\geq j^2\log(l_jl_{j-1}\dots l_1)$.
\item For every $B\in\mathcal A_j$, it contains at least $$C_0l_{j+1}^{\dim\mathfrak g_\alpha+2\dim\mathfrak g_{2\alpha}}\mu_{N_+}(B)$$ sub-open boxes of the form $B_{N_+}(r_0d(\tilde q)^{-\frac{2\alpha}{2\alpha-(\gamma+\epsilon)}}, r_0d(\tilde q)^{-\frac{4\alpha}{2\alpha-(\gamma+\epsilon)}})\tilde q$ with $$\tilde q\in S(B,l_{j+1}/2,l_{j+1}).$$
\end{enumerate}
For each $\tilde q\in S(B,l_{j+1}/2,l_{j+1})$, we devide $B_{N_+}(r_0d(\tilde q)^{-\frac{2\alpha}{2\alpha-(\gamma+\epsilon)}}, r_0d(\tilde q)^{-\frac{4\alpha}{2\alpha-(\gamma+\epsilon)}})\tilde q$ into cubes of side length $r_0d(\tilde q)^{-\frac{4\alpha}{2\alpha-(\gamma+\epsilon)}}$. We denote by $\mathcal A_{j+1}$ the collection of all these cubes, as $\tilde q$ runs through $S(B,l_{j+1}/2,l_{j+1})$ and $B$ runs over all the subsets in $\mathcal A_j$. Then define $\mathbf A_{j+1}=\bigcup\mathcal A_{j+1}$.

Now we take $\mathbf A_\infty=\bigcap_{j=0}^\infty\mathbf A_j$ and $\mathcal A=\bigcup_{j=0}^\infty\mathcal A_j$. By the construction of $\mathbf A_j$'s and Proposition \ref{cp67}, we know that $\mathbf A_\infty\subset S_\gamma^c\cap U_0(e\Gamma)$. Also we have $$\Delta_j(\mathcal A)\sim l_{j+1}^{\dim\mathfrak g_\alpha+2\dim\mathfrak g_{2\alpha}} l_{j+1}^{-\frac{2\alpha}{2\alpha-(\gamma+\epsilon)}\dim\mathfrak g_\alpha}l_{j+1}^{-\frac{4\alpha}{2\alpha-(\gamma+\epsilon)}\dim\mathfrak g_{2\alpha}}$$ and $$d_j(\mathcal A)=r_0 l_j^{-\frac{4\alpha}{2\alpha-(\gamma+\epsilon)}}.$$ By Theorem \ref{cthm1}, we calculate
\begin{eqnarray*}
&{}&\dim_H S_\gamma^c\cap U_0(e\Gamma)\geq\dim_H(\mathbf A_\infty)\\
&\geq&\dim\mathfrak g_\alpha+\dim\mathfrak g_{2\alpha}-\limsup_{j\to\infty}\frac{\sum_{i=0}^j\log\left(l_{i+1}^{-\frac{\gamma+\epsilon}{2\alpha-(\gamma+\epsilon)}\dim\mathfrak g_\alpha-\frac{2(\gamma+\epsilon)}{2\alpha-(\gamma+\epsilon)}\dim\mathfrak g_{2\alpha}}\right)}{\log\left(l_{j+1}^{-\frac{4\alpha}{2\alpha-(\gamma+\epsilon)}}\right)}\\
&=&\left(1-\frac{\gamma+\epsilon}{2\alpha}\right)\dim\mathfrak g_{2\alpha}+\left(1-\frac{\gamma+\epsilon}{4\alpha}\right)\dim\mathfrak g_\alpha.
\end{eqnarray*}
Let $\epsilon\to0$ and we have
\begin{eqnarray*}
\dim_H S_\gamma^c\cap U_0(e\Gamma)&\geq&\left(1-\frac\gamma{2\alpha}\right)\dim\mathfrak g_{2\alpha}+\left(1-\frac\gamma{4\alpha}\right)\dim\mathfrak g_\alpha.
\end{eqnarray*}
This completes the proof of the theorem.
\end{proof}

\section{Proof of Theorem \ref{ccor}}\label{pf2}
In this section we will prove Theorem \ref{ccor}. Let $$S(i,\gamma)=\{p\in G/\Gamma|\exists C>0 \text{ s.t. } \eta(a_tp)\chi_{Y_i}(a_tp)\geq Ce^{-\gamma t}\chi_{Y_i}(a_tp)\;(\forall t>0)\}.$$
By definition, we know that $$S_{\gamma_1,\dots,\gamma_k}=S(1,\gamma_1)\cap\cdots\cap S(k,\gamma_k)$$ and hence $$\dim_H S_{\gamma_1,\dots,\gamma_k}^c=\max_{1\leq i\leq k}\dim_H S(i,\gamma_i)^c.$$ So in order to prove Theorem \ref{ccor}, it is enough to prove the following

\begin{theorem}\label{ccortheorem}
Let $i\in\{1,\dots,k\}$ and $U$ be an open subset in $G/\Gamma$. If $\mathfrak g_{2\alpha}=0$, then the Hausdorff dimension of $S(i,\gamma)^c\cap U$ $(0\leq\gamma<\alpha)$ is $$\dim\mathfrak g_{-\alpha}+\dim\mathfrak g_0+\frac{\alpha-\gamma}\alpha\dim\mathfrak g_\alpha.$$
If $g_{2\alpha}\neq\emptyset$, then the Hausdorff dimension of $S(i,\gamma)^c\cap U$ $(0\leq\gamma<2\alpha)$ is
$$\dim\mathfrak g_{-2\alpha}+\dim\mathfrak g_{-\alpha}+\dim\mathfrak g_0+\frac{4\alpha-\gamma}{4\alpha}\dim\mathfrak g_\alpha+\frac{2\alpha-\gamma}{2\alpha}\dim\mathfrak g_{2\alpha}.$$
\end{theorem}

Now we will fix a cusp $\xi_i$ for some $i\in\{1,\dots,k\}$ and let $\sigma_i\in\Sigma$ be the element corresponding to $\xi_i$.
\begin{proof}[Proof of Theorem \ref{ccortheorem}]
The proof for $S(i,\gamma)^c\cap U$ is almost identical to the proof of Theorem \ref{cmaintheorem}, except that we replace rational points by $\sigma_i$-rational points, denominators by $\sigma_i$-denominators. In fact, our discussions in section \ref{crp} are cuspwise, and we can use Proposition \ref{cp53} and Proposition \ref{cp54} to count $\sigma_i$-rational points instead of Proposition \ref{cp56}. Hence Proposition \ref{cp71} and Proposition \ref{cp72} holds also for $\sigma_i$-rational points with $\sigma_i$-denominators. Same thing happens in Proposition \ref{cp64} and Proposition \ref{cp67}. The only thing we need to do is to prove that after replacing by $\sigma_i$-rational points and $\sigma_i$-denominators in Proposition \ref{cp62} and Proposition \ref{cp65} with the assumption that $p\in S(i,\gamma)^c$, the rational $q\in U_0(e\Gamma)$ we obtain is actually a $\sigma_i$-rational point. We will prove this for the case $\mathfrak g_{2\alpha}=0$. The case $\mathfrak g_{2\alpha}\neq0$ is similar.

Now assume $\mathfrak g_{2\alpha}=0$ and $p\in S({i,\gamma})^c$ ($a_{t_0}p\in Y_i$). By the proof of Proposition \ref{cp62}, we know that the rational point $q\in U_0(e\Gamma)$ satisfies the condition $$na_{t_0}p=a_{t_0}q$$ for some bounded $n\in N_+$ and we have $d(q)\sim e^{(\alpha-\gamma) t_0}$. This implies that $$d(a_{t_0}q)\sim e^{-\gamma t_0}$$ and the rational point $a_{t_0}q\in Y_i'$ for a small neighborhood $Y_i'$ of $\xi_i$ in $G/\Gamma$. This happens if and only if $q$ (being a rational point in $U_0(e\Gamma)$) is a $\sigma_i$-rational point.
\end{proof}

\section{Diophantine approximation in Heisenberg groups}\label{pf3}
In this section, we prove Theorem \ref{chptheorem}. First we recall some definitions and notations from \cite{HP3}.

Let $$Q(z_0,z_1,\dots,z_n)=-(z_0\overline{z_1}+z_1\overline{z_0})+z_2\overline{z_2}+\dots+z_n\overline{z_n}$$ which is a hermitian form defined on $\mathbb C^{n+1}$. Let $G\subset SL_{n+1}(\mathbb C)$ be the group preserving $Q(z_0,z_1,\dots,z_n)$ and $\Gamma=G(\mathbb Z[i])$. In the sequel, we will consider the homogeneous space $\Gamma\backslash G$. Note that we change the notation of the homogeneous space and the lattice $\Gamma$ now acts on the left side of $G$. The semisimple flow we will study is $$a_t=\left(\begin{array}{ccc}e^t & 0 & 0 \\0 & e^{-t} & 0 \\0 & 0 & I_{n-1}\end{array}\right)\quad ( t\in\mathbb R)$$ and it acts on $\Gamma\backslash G$ by right multiplication. Here $I_{n-1}$ denotes the $(n-1)$ by $(n-1)$ identity matrix. The notation for the adjoint action of $\{a_t\}$ on the Lie algebra $\mathfrak g$ of $G$ is now changed to $a_{-t }va_t $ $(v\in\mathfrak g)$. By section \ref{pre}, the corresponding subgroup $N_+$ of $\{a_t\}$ is equal to $$N_+=\left\{\left(\begin{array}{ccc}1& 0 & 0 \\v &1& \zeta^* \\ \zeta & 0 & I_{n-1}\end{array}\right):\zeta\in\mathbb C^{n-1},v\in\mathbb C,2\Re v=|\zeta|^2\right\}$$ which is isomorphic to the Heisenberg group $\mathcal H_{2n-1}(\mathbb R)$. Here $\zeta^*$ denotes the conjugate transpose of $\zeta$. The simple root $\alpha$ of $\{a_t\}$ is equal to 1.

We use the model of Siegel domain for the complex hyperbolic space $\mathbf H_{\mathbb C}^n$ (see section 3.8 in \cite{HP3}). Specifically, $$\mathbf H_{\mathbb C}^n=\{(1,w_1,w)\in\{1\}\times\mathbb C\times\mathbb C^{n-1}: 2\Re w_1-|w|^2>0\}$$ with the boundary $$\partial\mathbf H_{\mathbb C}^n=\{(1,w_1,w)\in\{1\}\times\mathbb C\times\mathbb C^{n-1}:2\Re w_1-|w|^2=0\}\cup\{\infty\}.$$ Here $\infty=(0,1,0)\in\mathbb C\times\mathbb C\times\mathbb C^{n-1}$. We will denote by $O=(1,0,0)\in\partial\mathbf H_{\mathbb C}^n$ and $o=(1,1,0)\in\mathbf H_{\mathbb C}^n$.

The group $G$ acts on $\mathbb C^{n+1}=\mathbb C\times\mathbb C\times\mathbb C^{n-1}$ simply by matrix multiplication, which induces an action of $G$ on $\mathbf H_{\mathbb C}^n$ by rescaling a vector $(w_0,w_1,w)$ to be $(1,w_1/w_0,w/w_0)$. The action of $G$ on $\mathbf H_{\mathbb C}^n$ extends naturally to an action on the boundary $\partial\mathbf H_{\mathbb C}^n$. In particular, for any $g\in N_+$ with $$g=\left(\begin{array}{ccc}1& 0 & 0 \\v &1& \zeta^* \\ \zeta & 0 & I_{n-1}\end{array}\right),$$ the action of $g$ on the boundary point $O$ is $$g.O=(1,v,\zeta).$$

Let $\mathcal M=\Gamma\backslash G/K\cong\Gamma\backslash\mathbf H_{\mathbb C}^n$ where $K$ is the stabilizer of $o=(1,1,0)$ in $G$. Then the semisimple flow $\{a_t\}$ on $\Gamma\backslash G$ corresponds to the geodesic flow on $\mathcal M$. We will denote by $\pi_K$ the projection from $\Gamma\backslash G$ to $\mathcal M$. For simplicity, we will assume that $\mathcal M$ has only one cusp $\xi=\infty$. For general case, one could essentially follow  \cite{Bor, HP3} and the arguments below.

We denote by $\mathcal L$ the set of geodesic lines starting from $\xi=\infty$ in $\mathcal M$. The geodesic lines in $\mathcal L$ which start from $\xi=\infty$ and diverge to $\xi=\infty$ are called rational lines, and the other geodesic lines in $\mathcal L$ are called irrational lines. Readers may refer to \cite{HP3} for the definitions of the height function $\beta$ on $\mathcal M$, the Hamenst\"adt distance $d_\infty$ on $\partial\mathbf H_{\mathbb C}^n$ and the depth $D(r)$ of a rational geodesic line $r$ in $\mathcal L$. 

In the following, we fix a small open subset $U_0\subset N_+$. Note that the orbit $\{(\Gamma e)a_t\}$ diverges to the cusp $\xi=\infty$ as $t\to\pm\infty$ since $\Gamma\cap N_-\neq\{e\}$ and $\Gamma\cap N_+\neq\{e\}$. This implies that for any $p\in (\Gamma e)U_0$, $\{pa_t\}$ also diverges to $\xi=\infty$ as $t\to-\infty$. Therefore $\pi_K(pa_t)$ is a geodesic line in $\mathcal L$.

\begin{lemma}\label{chp91}
Let $p\in (\Gamma e)U_0$ with $p=\Gamma g$ for some $g\in U_0$. Then $p$ is rational in the sense of Definition \ref{def51} if and only if $\pi_K(pa_t)$ is a rational geodesic line in $\mathcal L$ in the sense of \cite{HP3}.
\end{lemma}

\begin{proof}
Suppose that $p$ is rational in the sense of Definition \ref{def51}. Then by definition, $pa_t$ diverges in $\Gamma\backslash G$ to the cusp $\xi=\infty$ as $t\to\infty$ and hence so does $\pi_K(pa_t)$ in $\mathcal M$. Since $pa_t$ also diverges to the cusp as $t\to-\infty$. This implies that $\pi_K(pa_t)$ is a rational geodesic line in $\mathcal L$ in the sense of \cite{HP3}. 

Conversely, if $\pi_K(pa_t)$ is a rational geodesic line in $\mathcal L$, then $pa_t$ diverges to the cusp as $t\to\infty$ and hence by Corollary 6.2 in \cite{D0} and Proposition \ref{cp51}, we know that $p$ is rational in the sense of Definition \ref{def51}.
\end{proof}

\begin{lemma}\label{r91}
Let $p\in (\Gamma e)U_0$ with $p=\Gamma g$ for some $g\in U_0$. Then $p$ is rational in the sense of Definition \ref{def51} if and only if $g\in N_+\cong\mathcal H_{2n-1}(\mathbb R)$ is a rational point in $\mathcal H_{2n-1}(\mathbb Q)$.
\end{lemma}
\begin{proof}
Suppose that $g\in N_+\cong\mathcal H_{2n-1}(\mathbb R)$ is rational in $\mathcal H_{2n-1}(\mathbb Q)$. We know that $\Stab(p)=g^{-1}\Gamma g$ is commensurable with $\Gamma$. Hence $$\Stab(p)\cap N_-\neq\{e\}$$ and $p$ is rational by Definition \ref{def51}.

Conversely, suppose that $p$ is rational in the sense of Definition \ref{def51}. Then by Lemma \ref{chp91}, the orbit $\pi_K(pa_t)$ diverges to $\xi=\infty$ in $\Gamma\backslash\mathbf H_{\mathbb C}$, and hence $ga_t.o$ diverges to a rational point in the $\Gamma$-orbit of $\xi=\infty\in\partial\mathbf H_{\mathbb C}$. By calculations, this rational point is equal to $g.O$, where $O=(1,0,0)\in\partial\mathbf H_{\mathbb C}^n$. Hence $g$ is rational.
\end{proof}

\begin{lemma}\label{chp92}
Let $p\in (\Gamma e)U_0$ be rational in the sense of Definition \ref{def51} with $p=\Gamma g$ for some $g\in U_0$. Then we have $$h(g)\sim d(p).$$ Here by Lemma \ref{r91}, $g$ is a point in $\mathcal H_{2n-1}(\mathbb Q)$, $h(g)$ is the height of $g\in\mathcal H_{2n-1}(\mathbb Q)$ as in Definition \ref{def14} and $d(p)$ is the denominator of $p$ defined in the Definition \ref{def55}. The implicit constant depends only on $U_0$ and $G/\Gamma$.
\end{lemma}
\begin{proof}
In view of Lemma \ref{chp91}, we will consider the rational line $\pi_K(\{pa_t\})$ and its depth $D(\pi_K(\{pa_t\}))$. Fix a level set $\beta^{-1}(l)\subset\mathcal M$ for a sufficiently large $l>0$. Then there exists a constant $\epsilon>0$ such that for any $p'\in\pi_K^{-1}(\beta^{-1}(l))$ we have $$\eta(p')\sim\epsilon.$$ Let $s_0>0$ be the largest number with $pa_{s_0}\in \pi_K^{-1}(\beta^{-1}(l))$. Then by definition, we have 
\begin{equation}\label{cheq1}
\epsilon\sim\eta(pa_{s_0})\sim e^{-2s_0}d(p)^2.
\end{equation}
On the other hand, since $(\Gamma e)a_t$ diverges as $t\to-\infty$, there exists $t_0>0$ such that $t_0$ is the largest number with $(\Gamma e)a_{-t_0}\in\pi_K^{-1}(\beta^{-1}(l))$. This implies that 
\begin{equation}\label{cheq2}
\epsilon\sim\eta((\Gamma e)a_{-t_0})\sim e^{2t_0}. 
\end{equation}
As $p\in (\Gamma e)U_0$ and $a_{-t}$ contracts $U_0\subset N_+$ as $t\to\infty$, $pa_{-t_0}$ is near the subset $\pi_K^{-1}(\beta^{-1}(l))$. By definition of the depth function, this implies that there exists a constant $C_1>0$ such that $$|D(\pi_K(\{pa_t\}))-(s_0+t_0)|\leq C_1.$$ Also by equations (\ref{cheq1}) and (\ref{cheq2}), we know that $$d(p)^2\sim e^{2(s_0+t_0)}.$$ Hence $$d(p)\sim e^{D(\pi_K(\{pa_t\}))}.$$ By Proposition 3.14 in \cite{HP3}, for a rational geodesic line $r$ in $\mathcal L$, we have $$D(r)=\ln h(r).$$ Here, by Lemma \ref{chp91} and Lemma \ref{r91}, $r$ is identified with a point in $\mathcal H_{2n-1}(\mathbb Q)$ and $h(r)$ is the height of $r\in\mathcal H_{2n-1}(\mathbb Q)$. In our case, the geodesic $\pi_K(\{pa_t\})$ is identified with $g$. Hence we have $$d(p)\sim h(g).$$ Note that the implicit constant depends only on $U_0$ and $G/\Gamma$. This completes the proof of the lemma.
\end{proof}

Using the same arguments in the proof of Proposition \ref{cp61}, one could prove the following lemma.

\begin{lemma}\label{cl}
Let $p\in \Gamma\backslash G$ be a non-rational point. If $p$ is not Diophantine of type $\gamma$, then for any sufficiently small $\epsilon>0$ there exists a sequence $t_n\to\infty$ and $v_n$ satisfing the following conditions
\begin{enumerate}
\item $v_n$ is unipotent and $$\rho_0\epsilon e^{-\gamma l_n}\leq\eta(a_{l_n}p)\leq d_G(v_n,e)\leq3\kappa^3\rho_0\epsilon e^{-\gamma l_n};$$
\item if $v_n=\exp(v_{n,-}+v_{n,0}+v_{n,+})$ where $v_{n,-}\in\mathfrak n_-$, $v_{n,0}\in\mathfrak g_0$ and $v_{n,+}\in\mathfrak n_+$, then $$\frac{\rho_0}{3\kappa^3}\epsilon e^{-\gamma l_n}\leq\|v_{n,-}\|_{\mathfrak g}\leq\kappa^2\rho_0\epsilon e^{-\gamma l_n}.$$
\end{enumerate}
\end{lemma}

\begin{proposition}\label{cl0}
Let $p\in (\Gamma e)U_0$ be a non-rational point. Then $p$ is not Diophantine of type $\gamma$ if and only if for any sufficiently small $\epsilon>0$, there exists a sequence $q_n\in(\Gamma e)U_0$ of distinct rational points with $d(q_n)\to\infty$ such that $$p\in q_nB_{N_+}(C\epsilon^{\frac1{2-\gamma}}d(q_n)^{-\frac{2}{2-\gamma}}, C\epsilon^{\frac{2}{2-\gamma}}d(q_n)^{-\frac{4}{2-\gamma}}).$$ Here $C>0$ is a constant depending only on $G/\Gamma$ and $\gamma$.
\end{proposition}
\begin{proof}
The 'only if' part follows from Lemma \ref{cl} and the arguments in Proposition \ref{cp65} and Proposition \ref{cp66}. Now we prove the 'if' part. 

Suppose that for any sufficiently small $\epsilon>0$, there exists a sequence $q_n\in U(e\Gamma)$ of distinct rational points with $d(q_n)\to\infty$ such that $$p\in q_nB_{N_+}(C\epsilon^{\frac1{2-\gamma}}d(q_n)^{-\frac{2}{2-\gamma}}, C\epsilon^{\frac{2}{2-\gamma}}d(q_n)^{-\frac{4}{2-\gamma}}).$$ Let $t_n=\ln(\epsilon^{-\frac1{2-\gamma}}d(q_n)^{\frac2{2-\gamma}})$. Then one has $$pa_{t_n}\in (q_na_{t_n})B_{N_+}(C,C)$$ and 
\begin{align*}
\eta(q_na_{t_n})\leq& d(q_na_{t_n})^2=d(q_n)^2\epsilon^{\frac2{2-\gamma}}d(q_n)^{-\frac4{2-\gamma}}=\epsilon^{\frac2{2-\gamma}}d(q_n)^{-\frac{2\gamma}{2-\gamma}}\\
=&\epsilon e^{-\gamma t_n}.
\end{align*}
This implies that $$\eta(pa_{t_n})\leq C'\epsilon e^{-\gamma t_n}$$ for some $C'>0$. By Definition \ref{def12}, $p$ is not Diophantine of type $\gamma$.
\end{proof}

\begin{proposition}\label{chp93}
Let $p\in (\Gamma e)U_0$ with $p=\Gamma g$ for some $g\in U_0$. Then $p$ is Diophantine of type $2(1-1/\gamma)$ in the sense of Definition \ref{def12} if and only if $g\in N_+\cong\mathcal H_{2n-1}(\mathbb R)$ is Diophantine of type $\gamma$ in the Heisenberg group $\mathcal H_{2n-1}(\mathbb R)$.
\end{proposition}
\begin{proof}
By Proposition \ref{cl0}, $p$ is not Diophantine of type $2(1-1/\gamma)$ in $\Gamma\backslash G$ if and only if for any sufficiently small $\epsilon>0$, there exists a sequence $q_n\in (\Gamma e)U_0$ of distinct rational points with $d(q_n)\to\infty$ such that $$p\in q_nB_{N_+}(C\epsilon^{\frac\gamma2}d(q_n)^{-\gamma}, C\epsilon^{\gamma}d(q_n)^{-2\gamma}),$$ where $C>0$ is a constant depending only on $G$.

By the definition of the Cygan distance, Lemma \ref{r91} and Lemma \ref{chp92}, this is equivalent to saying that for any sufficiently small $\epsilon>0$, there exists a sequence of rational points $r_n$ in $\mathcal H_{2n-1}(\mathbb Q)$ with $h(r_n)\to\infty$ such that $$d_{\text{Cyg}}(g,r_n)\leq\epsilon^{\frac\gamma2}\frac C{(h(r_n))^\gamma}$$ for some constant $C>0$ depending only on $U_0$ and $G$, which means that $g\in\mathcal H_{2n-1}(\mathbb R)$ is not Diophantine of type $\gamma$.
\end{proof}

\begin{proof}[Proof of Theorem \ref{chptheorem}]
It is enough to compute the Hausdorff dimension of $L_\gamma^c\cap U_0$ for every small open subset $U_0$ in $N_+\cong\mathcal H_{2n-1}(\mathbb R)$. Now fix a small open subset $U_0$. By Proposition \ref{chp93}, we have $$\dim_H L_\gamma^c\cap U_0=\dim_H S_{2(1-1/\gamma)}^c\cap (\Gamma e)U_0,$$ and hence by Theorem \ref{ctheorem} $$\dim_H L_\gamma^c\cap U_0=\frac{1+\gamma}{2\gamma}2(n-1)+\frac1\gamma=\frac{1+\gamma}\gamma n-1.$$ This completes the proof of Theorem \ref{chptheorem}.
\end{proof}

\end{document}